\documentclass[reqno]{amsart}
\usepackage{amssymb,amsthm,amsfonts,amstext}
\usepackage{amsmath}
\usepackage{mathptmx}       
\usepackage{helvet}         
\usepackage{courier}        
\usepackage{pstricks}

\makeatletter \@addtoreset{equation}{section} \makeatother

\renewcommand\thefigure{\thesection.\@arabic\c@figure}
\renewcommand\thetable{\thesection.\@arabic\c@table}
\newtheorem{theorem}{Theorem}[section]
\newtheorem{lemma}[theorem]{Lemma}
\newtheorem{proposition}[theorem]{Proposition}
\newtheorem{corollary}[theorem]{Corollary}

\theoremstyle{definition}

\theoremstyle{remark}
\newtheorem{remark}[theorem]{Remark}

\newcommand{\mc}[1]{{\mathcal #1}}

\newcommand{\mb}[1]{{\mathbf #1}}
\newcommand{\bb}[1]{{\mathbb #1}}

\newcommand{\<}{\langle}
\renewcommand{\>}{\rangle}

\DeclareMathOperator{\Geom}{Geom}

\date{}

\title{Lattice model for Fast Diffusion Equation}
\author{F.~Hern\'andez, M.~Jara, F.~Valentim}

\address{\noindent 
Instituto de Matem\'atica, Universidade Federal Fluminense, Rua M\'ario Santos Braga S/N,
Niter\'oi, RJ 24020-140, Brazil
\newline
e-mail: \rm \texttt{freddyhernandez@id.uff.br}}

\address{\noindent IMPA, Estrada Dona Castorina 110, Rio de Janeiro, RJ, 22460-320, Brazil.
\newline
e-mail: \rm \texttt{mjara@impa.br}}

\address{
\noindent
Departamento de Matem\'atica, Universidade Federal do Esp\'irito Santo, Av. Fernando Ferrari, 514, Goiabeiras, Vit\'oria, 29075-910, Brazil.
\newline
e-mail \rm\texttt{fabio.valentim@ufes.br}
}
\begin{document}

\begin{abstract}

We obtain a \emph{fast diffusion equation} (FDE) as scaling limit of a sequence of zero-range process with symmetric unit rate. 
Fast diffusion effect comes from the fact that the diffusion coefficient 
goes to infinity as the density goes to zero. Therefore, in order to capture 
the behaviour for an arbitrary small density of particles,
we consider a proper rescaling of a model with a typically high number of particles per site.
 Furthermore, we obtain some results on the convergence for the method of lines for FDE.
\end{abstract}

\maketitle
\section{introduction}
Consider the one dimensional Cauchy problem
\begin{equation}
\label{fastm}
\left\{
\begin{array}{l}
{\displaystyle \partial_t u  \; =\; \partial_x ( u^{m-1} \partial_x \;  u ),
} 
\\
{\displaystyle u(0,\cdot) \;=\; u_0 (\cdot)},
\end{array}
\right.
\end{equation}
where $m \in \mathbb R$ and $u_0$ is an initial data. Note that cases $m=1$ and $m>1$ correspond to the linear  heat equation and porous medium equation, respectively. 

In case $m<1$, equation \eqref{fastm} is called \emph{fast diffusion equation} (FDE). This name comes from the fact that the diffusion coefficient  goes to infinity as the density goes to zero.  This equation has been extensively studied in the literature and arises in a number of different physical applications, see \cite{dk} and the references therein. For instance, when $m<0$, equation \eqref{fastm} provides a model for diffusion in plasma \cite{bh, lh}, appears in the study of cellular automata and interacting particle systems with self-organized criticality \cite{cv} and also describes a plane curve shrinking along the normal vector with speed depending on the curvature \cite{ew, gh}.

 Microscopic derivations of the heat and porous medium equations have already been obtained, see  \cite{EkhSep}, \cite{FenIscSep}, \cite{GonLanTon}  and \cite{KipLan}, for instance. 
  Here we provide a derivation of the hydrodynamic limit for a fast diffusion equation.
  We will restrict our attention to the case $m=-1$, imposing periodic boundary conditions as well as uniform positive initial data, namely,
\begin{equation}
\label{fast}
\left\{
\begin{array}{l}
{\displaystyle \partial_t u  \; =\;  - \Delta (1/u),
} 
\\
{\displaystyle u(0,\cdot) \;=\; u_0 (\cdot)}.
\end{array}
\right.
\end{equation}
where $u_0: \mathbb T \to \mathbb [\epsilon, \epsilon^{-1}]$ for some $\epsilon > 0$.
The purpose of this article is to define a family of conservative interacting particle systems whose
macroscopic density profile evolves according to the partial differential equation \eqref{fast}. Equilibrium fluctuations for this family of systems are studied in \cite{HJV}.

 It is well known, see for instance \cite{KipLan}, that the equation
 \begin{equation}
\label{NLheat}
\left\{
\begin{array}{l}
\partial_t u = \Delta \phi(u),\\
u(0,\cdot) = u_0(\cdot),\\
\end{array}
\right.
\end{equation}
where $\phi(\rho) = {\rho}/{(1+\rho)}$, can be obtained as a diffusive scaling limit of a zero-range process with symmetric unit rate, $g(k)=1$, if $k>0$ and $g(k)=0$ otherwise. A formal description of this process is the following.  Particles live on $\bb T_n$, the discrete one-dimensional torus with $n$ points. At each site of the lattice $\bb T_n$ there is a Poissonian clock of rate $2$. Each time the clock at site $x \in \bb T_n$ rings, one of the particles at this site moves to $x-1$ or $x+1$ with equal probability. We denote by 
$\eta = (\eta(x); x \in \bb T_n)$ the particle configurations and say that $\eta(x) \in \bb N_0$ is the  number of particles at site $x$ according to configuration $\eta$. This process is a continuous-time Markov chain and we denote it  by $\{\eta_t^n; t \geq 0\}$ and  by $\{S_n(t); t \geq 0\}$ the semi-group associated. The aforementioned dynamic conserves the number of particles and also is known to have a family of stationary product measures (\cite{Spi}), which can be indexed by the particle
density $\rho$ and denoted by $\bar\mu^n_{\rho}$.

A  rigorous derivation of the hydrodynamic equation \eqref{NLheat} can be obtained by means of  Yau's relative entropy method (\cite{OllVarYau}, \cite{Yau}). The general idea of this method is to show that the entropy $H_n(\mu^n S_n(tn^{2})| \bar\mu_{u(t,\cdot)}^n)$ between the  law of the process  $\{\eta_t^n; t \geq 0\}$ at time $t$, starting from  a fixed probability measure $\mu^n$, and the product measure with slowly varying parameter associated to the solution of the hydrodynamic equation, is relatively small. As a matter of fact, convergence of the empirical density, namely, 
\begin{equation}
\label{HL1}
\lim_{n \to \infty} \frac{1}{n} \sum_{x \in \bb T_n} \eta_t^n(x) F\big(\tfrac{x}{n}\big) = \int_{\bb T} u(t,x) F(x) dx,
\end{equation}
 can be derived from $H_n(\mu^n S_n(tn^{2})| \bar\mu_{u(t,\cdot)}^n) = \omicron(n)$, by using the entropy inequality. The validity of \eqref{HL1} for any time $t$ and any $F: \bb T \to \bb R$ continuous, is taken in many cases as the definition of hydrodynamic limit of a particle system.
 
It is worth pointing out that, although the hydrodynamic limit result mentioned above is valid for a wide class of rate functions $g$, the corresponding nonlinearity, namely, $ \Phi(\rho)=\mathbb E_{\bar\mu_{\rho}}[g(\eta(0))]$,  will always be an analytical function. In particular, if $m<1$ we can not obtain equation \eqref{fastm} as hydrodynamic limit from any zero-range process in the usual way. On the other hand, notice that we can produce the nonlinearity appearing in equation \eqref{fast} by fixing $\alpha \in (0,\infty)$, defining $\phi_n(\rho) = n^\alpha \phi(n^\alpha \rho)$ and noting that
$ \displaystyle
\lim_{n \to \infty} \big( \phi_n(\rho) -n^\alpha\big) = -\rho^{-1}.
$
This can be seen as a motivation to consider a family of processes with a typically large number of particles per site.

Let us denote by $\{u_t^n(x); t \geq 0, x \in \bb T_n\}$  the solution of the system of ODE's
\begin{equation*}
\tfrac{d}{dt} u_t^n(x) = \Delta_n \phi_n(u_t^n(x))
\end{equation*}
where $u_0^n$ is  initial data and $\Delta_n$ is the discrete Laplacian. We will show in Section \ref{s5} that there exists $T>0$ such that  the sequence $\{u^n\}_{n \in \mathbb N}$ uniformly converges on $[0,T]$ to the solution of \eqref{fast}.
Our main result, which is the content of Theorem \ref{t3}, is then obtained by showing that 
 \begin{equation*}
 H_n(\mu^n S_n(tn^{2+2\alpha})| \bar\mu_{n^{\alpha}u^n_t(\cdot)}^n) = \omicron(n), \ \ \text{for} \ \ t \in [0,T]. 
 \end{equation*}
In analogy to \eqref{HL1} it can be obtained, as a byproduct of Theorem \ref{t3}  and the uniform convergence of  the  sequence $\{u^n\}_{n \in \mathbb N}$, that
\begin{equation*}
\lim_{n \to \infty} \frac{1}{n^{1+\alpha}} \sum_{x \in \bb T_n} \eta_t^n(x) F\big(\tfrac{x}{n}\big) = \int_{\bb T} u(t,x) F(x) dx,
\end{equation*}
where $u$ is the solution of \eqref{fast}. 
Formally, we are giving a mass $n^{-(1+\alpha)}$ to each particle and putting a larger number of particles, of order $n^{\alpha}$, into the system. Observe that the hydrodynamic limit \eqref{HL1} corresponds to $\alpha =0$, on which case each particle has a mass ${n^{-1}}$.

One of the crucial ingredients needed in the proof of the hydrodynamic limit is the so-called {\em one-block estimate}, which in the case of zero-range processes, roughly states that spatial averages of $g$ over large microscopic boxes are asymptotically equivalent to the function $\Phi$, evaluated at the average number of particles. 
The main difficulty in carrying out this program in the present situation is that the density of particles per site  grows as a power of $n$. Because of that, a key compactness argument in the classical proof of the one-block estimate does not work. Here we get, following the approach proposed in \cite{JarLanSet} to establish a \emph{local replacement} limit, a quantitative proof of the one-block estimate which allows to circumvent the compactness argument by the use of the so-called {\em spectral gap inequality}, which gives a sharp bound on the largest eigenvalue of the dynamics restricted to a finite box. We strongly rely on the  spectral gap estimate for the zero-range process, obtained in \cite{Mor}. 

 Another important step in the application of the relative entropy method relies on the Laplace-Varadhan Theorem and some large deviations arguments for i.i.d. random variables. In the present setting, this part of the proof also differs from the usual one because of the dependence on $n$ of the parameters of the product measures $\bar\mu_{n^{\alpha}u^n_t(\cdot)}^n$. However, inspired in the well known fact that, in a proper scaling,  geometric distribution converges to an exponential distribution, we obtain some concentration inequalities which allow us to follow an alternative approach for this step. 

The article is organised as follows. 
In Section \ref{s1} we review some of the standard facts on hydrodynamic limit of zero-range processes, establish the notation and state the main results of the article.
Relying on the results obtained in the subsequent sections, we prove in Section \ref{s2} the main result of the
article following Yau's relative entropy method. The proof of the one-block  estimate as well as a concentration inequality
needed for the entropy method are the contents of Sections \ref{s4} and \ref{concentration}, respectively. Finally, Section \ref{s5} is dedicated to the study of convergence of the above mentioned discrete approximation for the FDE.

\section{Notation and results}
\label{s1}
\subsection{The model}
\label{s1.1.}
For each $n \in \bb N$,\footnote{
We will use the conventions 
$\bb N = \{1,2,3,\dots\}$ and $\bb N_0= \{0,1,2,\dots\}$
}
let $\bb T_n = \bb Z / n \bb Z$ be the discrete circle with $n$ points. We will think about $\bb T_n$ as a discrete approximation of the continuous circle $\bb T = \bb R / \bb Z$. Therefore, the parameter $n$ can be understood as a spatial scaling. The discrete circle $\bb T_n$ can be embedded into $\bb T$ by means of the canonical embedding $x \mapsto \frac{x}{n}$.

Let $\Omega_n = \bb N_0^{\bb T_n}$ be the state space of a continuous-time Markov chain to be described below. We denote by $\eta = (\eta(x); x \in \bb T_n)$ the elements of $\Omega_n$ and we call them {\em configurations}. We call the elements $x \in \bb T_n$ {\em sites} and we say that $\eta(x)$ is the {\em number of particles} at site $x$ according to configuration $\eta$. Define $g: \bb N_0 \to \bb R$ as $g(\ell) = 1$ for $\ell\neq 0$ and $g(0)=0$. Let $p(x,y); x,y \in \bb T_n$ be the transition rate of a simple symmetric random walk on $\bb T_n$, namely $p(x,y) = \mathbf{1} (|y-x|=1)$ for any $x,y \in \bb T_n$.\footnote{
We call $p(\cdot,\cdot)$ {\em transition rate} to point out that the underlying random walk evolves in continuous time. $1(A)$ denotes the indicator function of the set $A$.
}
We call this random walk the {\em underlying random walk}.
For $x,y \in \bb T_n$ and $\eta \in \Omega_n$ such that $\eta(x) \geq 1$, let $\eta^{x,y} \in \Omega_n$ be given by 
\[
\eta^{x,y}(z) = 
\begin{cases}
\eta(x)-1; &z=x\\
\eta(y)+1; &z =y \\
\eta(z); & z \neq x,y.
\end{cases}
\]
Fix $\alpha \geq 0$. This parameter will be related to the mass of each particle, and its precise meaning will be explained in a few lines. For $f : \Omega_n \to \bb R$ let us define $\mc L_n f: \Omega_n \to \bb R$ as
\[
\mc L_n f(\eta) = n^{2+2\alpha}\sum_{x,y \in \bb T_n} p(x,y) g(\eta(x)) \big\{ f(\eta^{x,y})-f(\eta) \big\} \: \: \: \: \: \: \footnote{
We adopt the convention $g(\eta(x))f(\eta^{x,y}) =0$ whenever $\eta(x)=0$.
}
\]
for any $\eta \in \Omega_n$. Notice the sub-diffusive time scaling $2+2\alpha$. By the definition of $p(\cdot,\cdot)$ the sum can be restricted to {\em neighboring} sites $x,y$, that is, to sites $x,y\in \bb T_n$ such that $|y-x|=1$. The {\em zero-range process} is the continuous-time Markov chain $\{\eta_t^n; t \geq 0\}$ with infinitesimal generator $\mc L_n$.
We denote by $\{S_n(t); t \geq 0\}$ the semigroup associated to this process. The dynamics of the zero-range process can be described as follows. On each site $x \in \bb T_n$ we put a Poissonian clock of rate $2 n^{2+2\alpha}$. Each time the clock of some site $x$ rings, we choose a neighbor $y$ of $x$ with uniform probability and we move a particle from $x$ to $y$. If there are no particles to move, nothing happens. Since the number of sites is finite, this dynamics is well defined for any $t \geq 0$. In other words, the chain {\em does not explode} in finite time.

Let $\mc D([0,\infty); \Omega_n)$ be the space of {\em c\`adl\`ag} paths in $\Omega_n$ equipped with the $J_1$-Skorohod topology. For a given probability measure $\mu$ on $\Omega_n$, we denote by $\bb P_{\mu}^n$ the law on the space $\mc D([0,\infty); \Omega_n)$ of the process $\{\eta_t^n; t \geq 0\}$ with initial distribution $\mu$. We denote by $\bb E_\mu^n$ the expectation with respect to $\bb P_\mu^n$.

\subsection{Invariant measures} 
\label{s1.2}

For each $n,k \in \bb N_0$, let us define
\[
\Omega_{n,k} = \Big\{ \eta \in \Omega_n; \sum_{x \in \bb T_n} \eta(x) = k\Big\}.
\]
The set $\Omega_{n,k}$ corresponds to the set of configurations with exactly $k$ particles. Notice that the total number of particles is preserved by the dynamics. Therefore, the sets $\Omega_{n,k}$ are left invariant by the dynamics. Since the underlying random walk is irreducible and $g(k) >0$ whenever $k \geq1$, the zero-range process is irreducible on each of the sets $\Omega_{n,k}$. The uniform measure $\mu_{n,k}$ in $\Omega_{n,k}$ turns out to be the unique invariant measure of the chain in $\Omega_{n,k}$. In fact, we can verify that the uniform measures $\mu_{n,k}$ satisfy the {\em detailed balance} equation. Combining these measure by means of a chemical potential $\theta \in [0,1)$, we see that the geometric product measures $\bar{\mu}_\theta^n$ given by
\[
\bar{\mu}_\theta^n(\eta) = \prod_{x \in \bb T_n} (1-\theta) \theta^{\eta(x)}
\] 
for any $\eta \in \Omega_n$ are also left invariant by $\{\eta_t^n; t \geq 0\}$, that is, $\bar{\mu}_\theta^n S_n(t)  = \bar{\mu}_\theta^n$ for any $t \geq 0$ and any $\theta \in [0,1)$. Notice that 
\[
\int g(\eta(x)) d \bar{\mu}_\theta^n = \theta, \quad \int \eta(x) d \bar{\mu}_\theta^n = \frac{\theta}{1-\theta}.
\]
Since the number of particles is the only quantity conserved by the dynamics, it is reasonable to parametrize the invariant measures $\bar{\mu}_\theta^n$ by the average number of particles. Define the function $\phi: [0,\infty) \to [0,1)$ as $\phi(\rho) = \frac{\rho}{1+\rho}$ for any $\rho \geq 0$. Notice that $\phi$ is the inverse of $\theta \mapsto \frac{\theta}{1-\theta}$.
We will give to each particle a mass $n^{-(1+\alpha)}$. Therefore, the bigger the $\alpha$ is, the larger the number of particles we are putting into the system. In order to have an average total mass $\rho$ we have to choose
\[
\theta_n(\rho) = \frac{\rho n^\alpha}{1+\rho n^\alpha}.
\]
Notice that the density of particles per site is equal to $ \rho n^\alpha$ under the measure $\mu_\rho^n$.
The canonical choice in the literature is $\alpha =0$, on which case each particle has a mass $\frac{1}{n}$ and $\rho$ can be interpreted as the average number of particles per site.
We will use the notation $\mu_{\rho}^n = \bar{\mu}_{\theta_n(\rho)}^n$.

\subsection{The hydrodynamic limit: the case $\alpha =0$}
\label{s1.3}

Given a function $u: \bb T \to [0,+\infty)$ we denote by $\nu_{u(\cdot)}^n$ the product measure in $\Omega_n$ given by
\begin{equation*}
\label{nu_t}
\nu_{u(\cdot)}^n(\eta) = \prod_{x \in \bb T_n} \frac{1}{1+n^\alpha u(\frac{x}{n})} \Big( \frac{n^\alpha u(\frac{x}{n})}{1+ n^\alpha u(\frac{x}{n})}\Big)^{\eta(x)},
\end{equation*}
that is, $\nu_{u(\cdot)}^n$ is a product of geometric distributions with expectations $n^\alpha u(\frac{x}{n})$.

Given two probability measures $\mu$, $\nu$ in $\Omega_n$, let $H_n(\mu|\nu)$ denote the relative entropy of $\mu$ with respect to $\nu$:
\begin{equation}
\label{ec1.3.2}
H_n(\mu|\nu) = \sup_{f} \Big\{ \int f d\mu - \log \int e^f d\nu\Big\},
\end{equation}
where the supremum runs over bounded functions $f: \Omega_n \to \bb R$. It turns out that $H_n(\mu|\nu) <+\infty$ implies that $\mu$ is absolutely continuous with respect to $\nu$, $\mu \ll \nu$, on which case we have the identity
\[
H_n(\mu|\nu) = \int \frac{d\mu}{d\nu} \log \frac{d \mu}{d \nu} d\nu.
\]
A very useful inequality involving entropy is obtained taking $\gamma f$ as a test function in \eqref{ec1.3.2}: for any $\gamma>0$ and any $f: \Omega_n \to \bb R$ integrable with respect to $\mu$, 
\begin{equation}
\label{entropy}
\int f d\mu \leq \frac{1}{\gamma} \Big\{ H_n(\mu | \nu) + \log \int e^{\gamma f} d \nu \Big\}.
\end{equation}
We call this inequality the {\em entropy inequality}.

The following result is well known (see Chapter 6 of \cite{KipLan} for instance):

\begin{theorem}
\label{t1}
Let $\alpha =0$ and $u_0: \bb T \to [0,+\infty)$ be a function of class $\mc C^{2+\delta}(\bb T)$ for some $\delta>0$.  Let $\{u(t,x); t \geq 0, x \in \bb T\}$ be the solution of the equation
\begin{equation}
\label{echid0}
\left\{
\begin{array}{l}
\partial_t u = \Delta \phi(u),\\
u(0,\cdot) = u_0(\cdot),\\
\end{array}
\right.
\end{equation}
where $\phi(\rho) =: \frac{\rho}{1+\rho}$.
Let $\{\mu^n; n \geq 1\}$ be a sequence of measures in $\Omega_n$ such that
\[
\lim_{n \to \infty} \frac{H_n\big(\mu^n| \nu_{u_0(\cdot)}^n\big)}{n} =0.
\]
Then, for any $t \geq 0$,

\[
\lim_{n \to \infty} \frac{H_n\big(\mu^n S_n(t)|\nu_{u(t,\cdot)}^n\big)}{n} =0
\]
\end{theorem}

This result is known in the literature as the {\em hydrodynamic limit} of the zero-range process, and the equation \eqref{echid0} is called the {\em hydrodynamic equation} associated to the zero-range processes $\{ \eta_t^n; t \geq 0\}$. What this result is telling us, is that the distribution of particles at time $t$ is close to a geometric product measure of averages $n^\alpha u(t,\frac{x}{n})$. In particular, the {\em density of particles} is well approximated by the solution of the hydrodynamic equation, in the sense of entropy. In fact, Theorem \ref{t1} has the following Corollary:

\begin{corollary}
\label{c1}
Under the hypothesis of Theorem \ref{t1}, for any $t \geq 0$ and any $F: \bb T \to \bb R$ continuous,
\[
\lim_{n \to \infty} \frac{1}{n^{1+\alpha}} \sum_{x \in \bb T_n} \eta_t^n(x) F\big(\tfrac{x}{n}\big) = \int_{\bb T} u(t,x) F(x) dx.
\]
\end{corollary}
This Corollary can be interpreted as a weak law of large numbers for the empirical density of particles of the process, and it is taken in many cases as the definition of hydrodynamic limit of a system of particles.
The smoothness assumption $u_0 \in \mc C^{2+\delta}(\bb T)$ is needed in order to ensure smoothness of the solution of the hydrodynamic equation \eqref{echid0}. 

Our main objective in these notes is to study the case $\alpha >0$. For each $\lambda >0$, let $u^\lambda$ denote the solution of \eqref{echid0} with initial condition $\lambda u_0$. At least formally, a simple computation shows that 
\[
u(t,x) = \lim_{\lambda \to \infty} \frac{u^\lambda(\lambda^2 t,x)}{\lambda}
\]
should be the solution of
\begin{equation}
\label{echid}
\left\{
\begin{array}{l}
\partial_t u = \partial_x \big( u^{-2} \partial_x u\big),\\
u(0,\cdot) = u_0(\cdot).
\end{array}
\right.
\end{equation}
The appearance of the factor $\lambda^2$ in the time scale explains the pre-factor $n^{2+2\alpha}$ in the definition of $\mc L_n$. We will show that for $\alpha \in (0,1)$  the equation \eqref{echid} is indeed the hydrodynamic equation associated to the processes $\{\eta_t^n; t \geq 0\}$. In order to state our result in a precise way, we need to introduce some additional definitions.
We also warn the reader to the fact that most of the definitions of here will be overrun by the definitions made in Section \ref{s1.5}.

\subsection{The discrete approximations}
\label{s1.4}

Here and everywhere we need to take a finite interval $[0,T]$, where $T$ is smaller than the explosion time of the second derivative (see Lemma \ref{l5.5}).
Although some steps  in the proof of the hydrodynamic limit can be done by means of more traditional methods, with an eye in future applications we want to get rid of the smoothness assumption $u_0 \in \mc C^{2+\delta}(\bb T)$. We will see below that it is very natural to consider {\em discrete approximations} of the hydrodynamic equation \eqref{echid}. For each $n \in \bb N$, define $\phi_n:[0,\infty) \to [0,\infty)$ as $\phi_n(u) = n^\alpha \phi(n^\alpha u)$. Notice that
\begin{equation}
\label{ec1.4.1}
\int n^\alpha g(\eta(x)) d \mu_\rho^n = \phi_n(\rho), \quad  \lim_{n \to \infty} \big( \phi_n(\rho) -n^\alpha\big) = -\rho^{-1}
\end{equation}
Notice as well that $\phi_n'(\rho) \to \frac{1}{\rho^2}$ as $n \to \infty$.

For $f: \bb T_n \to \bb R$ we define the discrete laplacian $\Delta_n f: \bb T_n \to \bb R$ as
\[
\Delta_n f(x) = n^2 \big( f(x+1) + f(x-1) -2 f(x)\big)\;.
\]

  Let $u_0^n: \bb T_n \to [0,\infty)$ be given. The natural choice in our context will be $u_0^n(x) = u_0(\frac{x}{n})$, where $u_0$ is the initial condition of the hydrodynamic equation \eqref{echid}. For each $n \in \bb N$, let $\{u_t^n(x); t \geq 0, x \in \bb T_n\}$ be the solution of the system of ODE's
\begin{equation}
\label{ec1.4.3}
\tfrac{d}{dt} u_t^n(x) = \Delta_n \phi_n(u_t^n(x))
\end{equation}
with initial data $u_0^n$. The following result is proved in Section \ref{s5}. 

\begin{theorem}
\label{t2}

Assume that $u_0\in \mc C^2(\bb T)$ and bounded below. That is, there exists $\epsilon >0$ such that
\[
u_0(x) \geq \epsilon \qquad \text{ for any } x \in \bb T.
\]
  Then, there exists $T > 0$ such that
\[
\lim_{n \to \infty} \sup_{x \in \bb T_n} \sup_{0 \leq t \leq T} \big|u_t^n(x) - u(t,\tfrac{x}{n})\big| =0,
\]
where $\{u(t,x); t \geq 0, x \in \bb T\}$ is the solution of the hydrodynamic equation \eqref{echid} with initial data $u_0$.
\end{theorem}

To make notation less cumbersome, we will write $\phi_t^n(x) = \phi_n(u_t^n(x))$ for  $t \geq 0$ and $x \in \bb T_n$.
Before state our main result, let us recall some general facts about attractiveness.

\subsection{Attractiveness and coupling inequality}

A well known property of zero-range processes with non-decreasing interaction rates  is its {\em attractiveness}. Let $\eta , \xi \in \Omega_n$, we say that $\eta \preceq \xi$ if $\eta(x) \leq \xi(x)$ for any $x \in \bb T_n$. The relation $\preceq$ defines a partial order in $\Omega_n$. Let $\mu$, $\nu$ be two probability measures in $\Omega_n$. We say that $\mu \preceq \nu$ if there exists a probability measure $\Pi$ in $\Omega_n \times \Omega_n$ such that $\Pi(\eta,\Omega_n) = \mu(\eta)$ and $\Pi(\Omega_n,\xi) = \nu(\xi)$ for any $\eta, \xi \in \Omega_n$ and such that
\[
\Pi\big((\eta,\xi) \in \Omega_n \times \Omega_n;  \eta \preceq \xi\big) =1.
\]
We say in that case that $\mu$ is {\em stochastically dominated} by $\nu$.

We say that a function $f: \Omega_n \to \bb R$ is {\em increasing} if $f(\eta) \leq f(\xi)$ whenever $\eta \preceq \xi$. The following proposition is actually an alternative definition of stochastic domination:

\begin{proposition}
\label{p1.4}
Let $\mu \preceq \nu$ be two probability measures in $\Omega_n$ and $f: \Omega_n \to \bb R$ be an increasing function such that $\int f d \mu >-\infty$, then $\int f d\mu \leq \int f d\nu$.
\end{proposition}

Notice that geometric distributions are stochastically ordered by their expectations. In particular, for any $u_1,u_2: \bb T_n \to [0,\infty)$ such that $u_1(x) \leq u_2(x)$ for any $x \in \bb T_n$ we have that
\[
\nu_{u_1(\cdot)}^n \preceq \nu_{u_2(\cdot)}^n.
\]

Now we are ready to say on which sense the zero-range process is attractive:

\begin{proposition}
\label{p1.5} Let $\eta^1 \preceq \eta^2$  be two initial configurations of particles in $\Omega_n$. There exists a Markov process $\{(\eta_t^{n,1}, \eta_t^{n,2}); t \geq 0\}$ defined on $\Omega_n \times \Omega_n$ such that $\{\eta_t^{n,i}; t \geq 0\}$ is a zero-range process with initial configuration $\eta^i$, $i=1,2$ and such that
\[
\eta_t^{n,1} \preceq \eta_t^{n,2} \text{ for any } t \geq 0. 
\]
\end{proposition} 

This proposition is what is known in the literature as the {\em attractiveness} of the zero-range process. A proof of this result can be found in \cite{And}. A simple consequence of this property is the following. Let $\mu \preceq \nu$, then $\mu S_n(t) \preceq \nu S_n(t)$ for any $t \geq 0$. 

We will use the attractiveness of the zero-range process in the following way:

\begin{proposition}
\label{p1.6}
Let $\{\mu^n; n\in \bb N\}$ be a sequence of probability measures in $\Omega_n$ such that there exists a constant $ \epsilon>0$ such that $\mu^n_\epsilon \preceq \mu^n$ for any $n\in \bb N$. Then, $\mu_\epsilon^n \preceq \mu^n S_n(t)$ for any $t \geq 0$. In particular, for any increasing function $f: \Omega_n \to \bb R$,
\[
\int f d \mu_\epsilon^n \; \leq \; \bb E^n_{\mu^n}[f(\eta_t^n)].
\]
\end{proposition}

\subsection{The hydrodynamic limit: the case $\alpha >0$}
\label{s1.5}

It turns out that the functions $\{u_t^n; n \in \bb N\}$ introduced in \eqref{ec1.4.3} are the right ones to construct the geometric product measures that serve as good approximations of the density of particles. Let $u_0$, $u_t^n$ as in Theorem \ref{t2} and define $\nu_t^n = \nu_{u_t^n(\cdot)}^n$. The main result of this manuscript is the following.

\begin{theorem}
\label{t3}
Fix $\alpha \in (0,1)$. Let $\{\mu^n; n \in \bb N\}$ be a sequence of probability measures in $\Omega_n$ such that 
\[
\lim_{n \to \infty} \frac{H_n(\mu^n|\nu_0^n)}{n} =0.
\]
Assume as well that there exists a constant $\epsilon >0$ such that $\mu_\epsilon^n \preceq \mu^n \preceq \mu_{\epsilon^{-1}}^n \ $ 
for any $n \in \bb N$.
Then, for any $0 \leq t \leq T$ we have that
\[
\lim_{n \to \infty} \frac{H_n(\mu^n S_n(t)|\nu_t^n)}{n} =0,
\]
where $T$ is the same constant appearing in Theorem \ref{t2}.
\end{theorem}
Notice that Theorem \ref{t2} implies that
\[
\lim_{n \to \infty} \frac{H_n(\nu_t^n| \nu_{u(t,\cdot)}^n)}{n} =0.
\]
Therefore, {\em a posteriori} we could have stated this theorem in terms of $\nu_{u(t,\cdot)}^n$. However, we want to emphasise that the measures $\nu_{u_t^n(\cdot)}^n$ are more natural as reference measures, a fact that can be useful in other situations. This situation is very common in homogenization theory, and the trick used here can be thought as a simple version of {\em compensated compactness}.

The restriction $\alpha \in (0,1)$ comes from the method we use in order to prove Theorem \ref{t3}, and it is not intrinsic to the problem. We strongly rely on the {\em spectral gap estimate} for the zero-range process, obtained in \cite{Mor}. In principle this restriction could be relaxed if more powerful techniques were available, like logarithmic Sobolev inequalities, but we did not pursue that line of reasoning further away.

\section{The relative entropy method}
\label{s2}

In this section we outline how do we prove Theorem \ref{t3} using Yau's relative entropy method \cite{Yau}. We will use various lemmas most of which will be proven in subsequent sections. Assuming the validity of these lemmas, our outline gives a rigorous proof of Theorem \ref{t1}.

Fix $\rho >0$ and recall the definition of $\mu_\rho^n$ as the geometric product measure with density of particles $n^\alpha \rho$. Let $\{\mu^n;\; n \in \bb N\}$ be as in Theorem \ref{t3} and let us write $\bb E^n = \bb E^n_{\mu^n}$.  Define
\[
f_t^n = \frac{d  \mu^n S_n(t)}{d \mu_{\rho}^n}, \quad \Upsilon_t^n = \frac{d \nu_t^n}{d \mu_\rho^n}\qquad \text{and} \qquad\;\;H_n(t) = H ( \mu^n S_n(t)| \nu_t^n)\; .
\]
 Yau's entropy inequality \cite{Yau} states that
\begin{equation}
\label{yau}
\tfrac{d}{dt} H_n(t) \leq \int \frac{f_t^n}{\Upsilon_t^n} \Big( n^{2+2\alpha} \mc L_n^\ast \Upsilon_t^n - \partial_t \Upsilon_t^n \Big) d \mu_\rho^n,
\end{equation}
where $\mc L_n^\ast $ is the adjoint of $\mc L$ in $L^2(\mu_\rho^n)$. This inequality does not rely on the particular form of the measures $\nu_t^n$, which in principle can be changed according to the needs of the model. In the model considered here $\mc L_n^\ast = \mc L_n$, but this point is not very important. The invariance of the measure $\mu_\rho^n$ under the evolution of the system is however crucial for the method. Notice that 
\[
\Upsilon_t^n =  \prod_{x \in \bb T_n} \frac{1+n^\alpha \rho}{1+n^\alpha u_t^n(x)} \Big( \frac{\phi_t^n(x)}{\phi_n(\rho)}\Big)^{\eta(x)}.
\]
After some long but standard computations (see \cite[chapter 6]{KipLan}), we have that
\[
\frac{n^{2+2\alpha} \mc L_n \Upsilon_t^n}{\Upsilon_t^n} = \sum_{x \in \bb T_n} \big\{ n^{2\alpha} g(\eta(x)) - n^\alpha \phi_t^n(x)\big\} \frac{\Delta_n \phi_t^n(x)}{\phi_t^n(x)},
\]
\[
\frac{\partial_t \Upsilon_t^n}{\Upsilon_t^n} = \sum_{x \in \bb T_n} \big\{ \eta(x) - n^\alpha u_t^n(x)\big\} \frac{\partial_t \phi_t^n(x)}{\phi_t^n(x)}.
\]
Using the fact that $\partial_t \phi_t^n(x) = \phi_n'(u_t^n(x)) \Delta_n \phi_t^n(x)$ we see that
\begin{multline*}
\frac{1}{\Upsilon_t^n} \big( n^{2+2\alpha}\mc L_n \Upsilon_t^n- \partial_t \Upsilon_t^n\big) 
		= n^\alpha \sum_{x \in \bb T_n} \frac{\Delta_n \phi_t^n(x)}{\phi_t^n(x)} \Big\{ n^\alpha g(\eta(x)) - \phi_t^n(x) - \\
		-\phi_n'(u_t^n(x)) \big(n^{-\alpha}\eta(x) - u_t^n(x)\big)\Big\}.
\end{multline*}
Recall  in \eqref{ec1.4.1} that $\int n^\alpha g(\eta(x)) d \mu_\rho^n = \phi_n(\rho)$. For each $n \in \bb N$, define $F_n: \bb T_n \times [0,\infty) \to \bb R$ as 
\begin{equation}
\label{Fn}
F_n(x,t) = \frac{n^\alpha \Delta_n \phi_t^n(x)}{\phi_t^n(x)}.
\end{equation}
At this point we can explain why Theorem \ref{t1} needs to assume that $u_0 \in \mc C^{2+\delta}(\bb T)$. In that case, the solution of the hydrodynamic equation is of class $\mc C^{2+\delta}$, uniformly in time. Therefore, by standard methods in numerical analysis, it can be checked that $F_n(x,t)$ is uniformly bounded in $n$, $x$ and $t$. And this property is a cut point for the method: if for some reason we know {\em a priori} that $F_n$ is uniformly bounded, then we can go on with the relative entropy method without further reference to the smoothness of the solution of the hydrodynamic equation \eqref{echid0}. It is proved in Section \ref{s5}, under the assumption that $u_0 \in \mc C^2(\bb T)$, a weaker property, namely that $F_n(x,t)$ stays bounded for a positive amount of time $T$, uniformly in $n, x$ and $t \leq T$:

\begin{lemma}
\label{l2.1}
Let $u_0: \bb T \to [0,\infty)$ be strictly positive and of class $\mc C^2(\bb T)$. Let $\{u_t^n(x); t \geq 0, x \in \bb N\}$ be the solution of \eqref{ec1.4.3} with initial condition $u_0^n(x) = u_0(\frac{x}{n})$. There exist constants $T >0$ and $K <+\infty$ depending only on $\|u_0\|_\infty$ and $\|\Delta u_0\|_\infty$ such that 
\[
\sup_{n \in \bb N} \sup_{x \in \bb T_n} \sup_{0\leq t \leq T} \big| F_n(x,t)\big| \leq K.
\]
\end{lemma}
  
  \vspace{0.5cm}

Integrating Yau's entropy inequality \eqref{yau} in time we see that
\begin{multline}
\label{ec2.1}
H_n(t) \leq H_n(0) + \bb E^n\Big[ \int_0^t \sum_{x \in \bb T_n} F_n(x,t) \big\{ n^\alpha g(\eta_s^n(x)) - \phi_n(u_s^n(x)) -\\
- \phi_n'(u_s^n(x)) \big(n^{-\alpha} \eta_s^n(x) - u_s^n(x)\big)\big\}ds\Big].
\end{multline}
The idea is to bound this integral by $\frac{1}{\beta}\int_0^t H_n(s) ds$ plus a term of order $\omicron(n)$. 
If we are able to do this, Theorem \ref{t3} will follow after the use of Gronwall's inequality and a concentration inequality. A key step will be the use of the entropy inequality \eqref{entropy}.
But before using the entropy inequality we need to replace the function $g(\eta_t^n(x))$ by a function that concentrates around its mean with respect to the measure $\nu_t^n$.

The first step into this program is what is known as the {\em one-block estimate}. Before stating this estimate, we need to introduce some definitions. For $\ell \leq n$ in $\bb N$, $x \in \bb T_n$ and $t \geq 0$ we define
\begin{equation}
\label{etanl}
\eta_t^{n,\ell}(x) = \frac{1}{n^\alpha\ell} \sum_{i=1}^\ell \eta_t^n(x+i).
\end{equation}
In other words, $\eta_t^{n,\ell}(x)$ is the density of particles on a box of size $\ell$ at the right of $x \in \bb T_n$, normalized by $n^\alpha$.
\begin{lemma}
\label{l1}
Let $F_n: \bb T_n \times [0,\infty) \to \bb R$ be defined as above. For $\alpha, \delta >0$ satisfying $ 2 \alpha+3 \delta < 2$, we have
\[
\lim_{n \to \infty} n^{-1} \bb E^n\Big[ \Big| \int_0^t \sum_{x \in \bb T_n} F_n(x,s) \big\{ n^\alpha g(\eta_s^n(x)) - \phi_n(\eta_s^{n,\ell}(x))\big\} ds\Big| \Big] =0
\]
uniformly in $t \in [0,T]$, where $\ell = \ell_n = n^\delta$.
\end{lemma}

The preceding result will be proved in Section \ref{s4}. Using this lemma, we see that we only need to bound the expectation
\[
\bb E^n\Big[ \int_0^t \sum_{x \in \bb T_n} F_n(x,t) \big\{ \phi_n(\eta_s^{n,\ell}(x)) - \phi_n(u_s^n(x))
- \phi_n'(u_s^n(x)) \big( n^{-\alpha}\eta_s^n(x) -u_s^n(x)\big)\big\}ds\Big]. 
\]

The advantage of this expectation with respect to the one appearing in \eqref{ec2.1} is that we have introduced a function of the density of particles, which we know it concentrates around its mean with respect to the measures $\nu_t^n$. We need to do something similar with the function $\eta_s^n(x)$:

\begin{lemma}
\label{l2}
Let $F_n$ be as above and $\ell = \ell_n = n^\delta$ for $0<\delta<1$. Then
\begin{equation}
\label{1halfblock}
\lim_{n \to \infty}  \bb E^n \Big[ \Big| \int_0^t \frac{1}{n} \sum_{x \in \bb T_n} F_n(x,s) \phi_n'(u_s^n(x))\big\{ n^{-\alpha}\eta_s^n(x) -  \eta_s^{n,\ell}(x)\big\} ds \Big| \Big] =0
\end{equation}
uniformly in $t \in [0,T]$.
\end{lemma}
To prove this result first perform a summation by parts replacing the expectation in  \eqref{1halfblock} by
$$ \bb E^n \Big[ \Big| \int_0^t \frac{1}{n^{1+\alpha}} \sum_{x \in \bb T_n}\eta_s^n(x)\frac 1\ell\sum_{y=1}^\ell \big\{ F_n(x,s) \phi_n'(u_s^n(x)) -   F_n(x-y,s) \phi_n'(u_s^n(x-y))\big\} ds \Big| \Big], $$
and then use the next result which follows directly from Remark \ref{energy2} in Section \ref{s5}.

\begin{lemma}
\label{l3b}Assume $u_0 \in \mc C^2(\bb T)$. 
Then there exists a finite constant $K$ such that
\[
\big|F_n(y,t) \phi_n'(u_t^n(y)) - F_n(x,t) \phi_n'(u_t^n(x))\big| \leq K\big|\tfrac{y-x}{n}\big|^{\frac{1}{2}},
\]
for any $n \in \bb N$, any $x,y \in \bb T_n$ and any $0 \leq t \leq T$.
\end{lemma}
 
Using Lemma \ref{l2}, we just need to bound the expectation
\begin{equation}
\label{ec2.3}
\bb E^n\Big[ \int_0^t \sum_{x \in \bb T_n} F_n(x,t) \big\{ \phi_n(\eta_s^{n,\ell}(x)) -\phi_n(u_s^n(x))
-\phi_n'(u_s^n(x)) \big( \eta_s^{n,\ell}(x) -u_s^n(x)\big)\big\}ds\Big].
\end{equation}
For $\ell \leq n$ in $\bb N$ and $x \in \bb T_n$, define
\[
u_t^{n,\ell}(x) = \frac{1}{\ell} \sum_{i=1}^\ell u_t^n(x+i).
\]
Notice that under $\nu_t^n$ the expectation of $\eta_t^{n,\ell}(x)$ is equal to $u_t^{n,\ell}(x)$ and not $u_t^n(x)$. Therefore, it seems to be a good idea to replace $u_t^n(x)$ by $u_t^{n,\ell}(x)$ in \eqref{ec2.3}. This is accomplished using the following result which is a straightforward consequence of Lemma \ref{l5.4} proved in Section \ref{s5}.

\begin{lemma}
\label{l3}Assume $u_0 \in \mc C^1(\bb T)$. 
Then there exists a finite constant $K_0$ such that
\[
\big|u_t^n(y) - u_t^n(x)\big| \leq K_0\big|\tfrac{y-x}{n}\big|^{\frac{1}{2}},
\]
for any $n \in \bb N$, $x,y \in \bb T_n$ and  $t \geq 0$.
\end{lemma}
Using this lemma and the coupling inequality 
we see that we are left to obtain a convenient bound for the expectation
\begin{equation}
\label{ec2.3.1}
\bb E^n\Big[\int_0^t  \sum_{x \in \bb T_n} F_n(x,s) M_n(u_s^{n,\ell}(x), \eta_s^{n,\ell}(x)) ds\Big],
\end{equation}
where $M_n: \bb R \times \bb R \to \bb R$ is defined as
\[
M_n(u,v) = \phi_n(v) -\phi_n(u) -  \phi_n'(u)(v-u).
\]
Let us summarize what we have accomplished up to here. We have proved that
\begin{equation}
\label{ec2.3'}
H_n(t) \leq H_n(0) + \bb E^n\Big[\int_0^t \sum_{x \in \bb T_n} F_n(x,s) M_n(u_s^{n,\ell}(x), \eta_s^{n,\ell}(x)) ds \Big] + R_n(t),
\end{equation}
where
\[
\lim_{n \to \infty} \frac{R_n(t)}{n} =0.
\]
Notice that
\[
M_n(u,v) = - \frac{(v-u)^2}{(n^{-\alpha}+u)^2(n^{-\alpha}+v)}.
\]
In particular, $M_n$ is singular near zero. Therefore, in order to obtain appropriate bounds for the integral term in \eqref{ec2.3'}, it will be necessary to rule out small densities of particles. This is accomplished by the next lemma.

Before of that, let us recall that an i.i.d. sequence of exponential random variables with mean $\rho$ satisfies the large deviation principle with rate function 
\begin{equation}
\label{ldprateexp}
 I_\rho(a) = \frac a\rho - 1- \log\frac a\rho.
\end{equation}
Section \ref{concentration} is devoted to obtain some concentration inequalities for $\eta_t^{n,\ell}(x)$ in terms of $I_{\rho}(\cdot)$.

\begin{lemma}
\label{l4}
For any $0<\epsilon_0 <\epsilon$ and any $n \in \bb N$,
\begin{equation}
\label{ec.l4}
\bb E^n\Big[\int_0^t  \sum_{x \in \bb T_n} F_n(x,s) M_n(u_s^{n,\ell}(x), \eta_s^{n,\ell}(x)) \mathbf{1}( \eta_s^{n,\ell}(x) \leq \epsilon_0) ds\Big] \leq t \|F_n\|_\infty n^{1+\alpha} e^{-n^\delta I_{\epsilon}(\epsilon_0)},
\end{equation}
uniformly in $t \in [0,T]$, , where $\ell = \ell_n = n^\delta$.
\end{lemma}
\begin{proof}
Thanks to the introduction of the indicator function, we can assume that $\eta_s^{n,\ell}(x) \leq \epsilon_0$ and $u_s^{n,\ell}(x) \geq \epsilon$. In that case, we have the bound $|M_n| \leq n^\alpha$. Notice that the function $\mathbf{1}(\eta_t^{n,\ell}(x) \leq \epsilon_0)$ is decreasing in $\eta$. Therefore, from Proposition \ref{p1.6} and Lemma \ref{l3.1}, we can see that the left-hand side of \eqref{ec.l4} is bounded above by
\begin{equation*}
 t \|F\|_\infty  n^{1+\alpha}  \mu_\epsilon^n\big( \eta^{n,\ell}(x) \leq \epsilon_0\big) \leq 
 	t \|F_n\|_{\infty} n^{1+\alpha} e^{-\ell I_\epsilon(\epsilon_0)} = t \|F_n\|_\infty n^{1+\alpha} e^{-n^\delta I_{\epsilon} (\epsilon_0)}.
\end{equation*}
\end{proof}

This lemma is telling us that we can introduce the indicator function ${\mathbf 1} ( \eta_s^{n,\ell}(x) \geq \epsilon_0)$ into the expectation \eqref{ec2.3.1}, effectively cutting off regions with small density of particles.

The following lemma relates the function $M_n(u,v)$ to the large deviations rate function $I_\rho(a)$.

\begin{lemma}
\label{l5} For any $\epsilon,\; \epsilon_0 >0$ there exists a positive constant $C= C(\epsilon, \epsilon_0)$ such that 
\[
\big|M_n(u,v)\big| \leq C I_u(v) 
\]
for any $u \in [\epsilon ,\epsilon^{-1}]$ and any $v \geq \epsilon_0$.
\end{lemma}
\begin{proof}
Notice that $\big|M_n(u,v)\big| \leq \frac{(v-u)^2}{u^2v} = u^{-1} \mc{M}(\frac{v}{u})$, where $\mc M(x) =: \frac{(1-x)^2}{x}$. Therefore, it is enough to show that
\[
 \mc M(v/u) \leq C u I_u(v) = C u I_1({v}/{u}).
\]
If $v \geq u$ note that 
\[
\mc M(x) \leq \min\{(x-1)^2, x-1\} \leq 4 I_1(x),
\]
 for any $x \geq 1$. On the other hand, if $u \geq v$, use that $I_{1}(x) \geq \frac{1}{2}(1-x)^2$ for  any $x \in (0,1)$, to obtain 
\[
\mc M(x) \leq \tfrac{1}{\epsilon_0 \epsilon}(1-x)^2 \leq  \tfrac{2}{\epsilon_0\epsilon} I_{1}(x).
\]
Thus, the desired result follows by taking $C= \max\{ \frac4{ \epsilon}, \frac{2 }{\epsilon_0\epsilon^2}\}$.
\end{proof}

Now, we are in position to bound the integral term in \eqref{ec2.3'}.

\begin{lemma}
Given $\epsilon> \epsilon_0>0$, there exists $\beta >0$ small enough such that for any $0\leq t \leq T$,
\begin{equation*}
\bb E^n\Big[\int_0^t  \sum_{x \in \bb T_n} F_n(x,s) M_n(u_s^{n,\ell}(x), \eta_s^{n,\ell}(x)) \mathbf{1}( \eta_s^{n,\ell}(x) > \epsilon_0) ds\Big] \leq \beta^{-1} \int_0^t H_n(s) ds + \omicron(n),
\end{equation*}
where $\ell = \ell_n = n^\delta$.
\end{lemma}

\begin{proof}
From the entropy inequality, we have
\begin{multline}
\label{ec1.l5}
\bb E^n\Big[ \sum_{x \in \bb T_n} F_n(x,s) M_n(u_s^{n,\ell}(x), \eta_s^{n,\ell}(x)) \mathbf{1}( \eta_s^{n,\ell}(x) >\epsilon_0)\Big] \leq  
		\beta^{-1} H_n(s) + \\
				\beta^{-1} \log E_s^n\Big[ \exp\big\{ \beta \sum_{x \in \bb T_n} F_n(x,s) M_n(u_s^{n,\ell}(x), \eta^{n,\ell}(x)) \mathbf{1}( \eta^{n,\ell}(x) >\epsilon_0)\big\} \Big],
\end{multline}
for any $\beta >0$, where $E_s^n$ denotes the expectation with respect to the measure $\nu_{u_s^n(\cdot)}^n$.

Our aim is to obtain an upper bound of order $\omicron (n)$ for the last term in the previous expression. Since $\eta^{n,\ell}(x)$ and  $\eta^{n,\ell}(y)$ are independent under $\nu_{u_s^n(\cdot)}^n$ as soon as $|y-x| \geq \ell$, by H\"older inequality, the last term in the right-hand side of \eqref{ec1.l5} is bounded by
\[
 \frac{1}{\beta \ell} \sum_{x \in \bb T_n} \log E_{s}^n\big[\exp\big\{ -\beta K \ell M_n\big(u_s^{n,\ell}(x),\eta^{n,\ell}(x)\big){\mathbf 1}(\eta^{n,\ell}(x) > \epsilon_0)\big\}\big],
\]
where $K$ is as in Lemma \ref{l2.1}.
 Since $u_s^{n,\ell}(x) \in [\epsilon,\epsilon^{-1}]$ (see Corollary \ref{c5.1}), we obtain that the previous  expression is bounded above by
\begin{equation}
\label{ec4}
 \frac{1}{\beta \ell} \sum_{x \in \bb T_n} \log E_s^n \big[ \exp\big\{ \ell \kappa  I_{u_s^{n,\ell}(x)}(\eta^{n,\ell}(x)){\mathbf 1}(\eta^{n,\ell}(x) > \epsilon_0)\big\} \big],
\end{equation}
where $\kappa = \beta KC_0 $ and $C_0=C_0(\epsilon,\epsilon_0)$ is the constant appearing in Lemma \ref{l5}.

For the convenience of the reader, we introduce at this point some notation compatible with the one used in Section \ref{concentration}. Namely, $\rho_{n,\ell,x,s} = u^{n, \ell}_s(x)$,  
$\rho^+_{n,\ell,x,s}= \max \{u_s^n(x+i);  \ i=1,\dots,\ell \}$, $\rho^-_{n,\ell,x,s} = \min\{u_s^n(x+i); \ i=1,\dots,\ell\}$ and $K^+_{n,\ell,x,s} = (\frac{\rho^+_{n,\ell,x,s}}{\rho_{n,\ell,x,s}})^2$. For simplicity, in the sequel we will omit subindexes for the notation introduced in this paragraph.

In order to bound the expectation inside of \eqref{ec4},
\begin{equation}
\label{ec2.l5}
E_s^n \big[ \exp\big\{ \ell \kappa  I_{u_s^{n,\ell}(x)}(\eta^{n,\ell}(x)){\mathbf 1}(\eta^{n,\ell}(x) > \epsilon_0)\big\} \big]\, ,
\end{equation}
we will consider separately the cases where $\eta^{n,\ell}(x)$ belongs to the intervals $[\epsilon_0, \rho^{-})$, $[\rho^{-}, K^{+}\rho)$ and $[K^{+}\rho, +\infty)$. But before that, let us
recall the following elementary fact:
for any non-negative function $f$ and any  discrete random variable $Z$ taking values in  $\{d_0 < d_1< \cdots\} $,
\begin{equation}
\label{dexp1}
E[f(Z)] =  f(d_0) + \sum_{j=1}^{\infty} [f(d_j) - f(d_{j-1})] P[Z \geq d_j ] \ .
\end{equation}

{\bf Case $[\epsilon_0, \rho^{-})$:} We use the previous  identity with $d_j=j/\ell n^{\alpha}$, $-\rho^- \leq d_j \leq -\epsilon_0$, $Z=- \eta^{n,\ell}(x)$ and $f(Z) = \exp\{\ell \kappa I_{\rho} (Z) \}$. Estimate \eqref{ec3.2} together with the mean value theorem permit to conclude that 
\begin{eqnarray*}
E_s^n \big[ \exp\big\{ \ell \kappa  I_{\rho}(\eta^{n,\ell}(x))\big\}{\mathbf 1}(\eta^{n,\ell}(x) \in [\epsilon_0, \rho^-)) \big]
\end{eqnarray*}
is  bounded above by
$$
e^{\ell\kappa I_{\rho}(\rho^-)} +  \frac{\ell \kappa(\rho^- - \epsilon_0)}{\epsilon_0} \exp\Big(\frac{\ell \rho^-}{n^\alpha\epsilon_0^2}\, +\, \ell \max_{\epsilon_0 < z < \rho^-} \{ \kappa  I_{\rho}( z) -  I_{\rho^-}(z)\} \Big).
$$
In view of Lemma \ref{Irhopsi2}, last expression is bounded by
\begin{equation}
\label{inter2a}
e^{\ell\kappa I_{\rho}(\rho^-)} + \frac{\ell \kappa(\rho^- - \epsilon_0)}{\epsilon_0}
\exp\Big( \frac{\ell\rho^-}{n^\alpha\epsilon_0^2} \, +\, \ell\kappa(1-\kappa /2)
\Big(\frac{\rho - \rho^-}{\rho - \kappa \rho^-}\Big)^2\Big)\, .
\end{equation}

{\bf  Case $[\rho^{-}, K^+\rho)$ :} In this interval the expectation in \eqref{ec2.l5} is bounded above by 
\[
\max \big\{ \exp\big(\ell \kappa I_{\rho}(\rho^{-} ) \big) , \exp\big(\ell\kappa I_{\rho}(K^+\rho ) \big)\big\}
\]
 which by the first part of Lemma \ref{Irho} is bounded above by
\begin{equation}
\label{inter1}
\max\big\{ \exp\big(\ell\kappa \rho^{-2}[\rho^{-} - \rho]^2 \big) , \exp\big(\ell\kappa [K^{+} -1]^2 \big)\big \}.
\end{equation}

{\bf Case $[K^+\rho, +\infty)$ :} Combining \eqref{dexp1} with estimate \eqref{ec3.1}, we see that in this case the expectation in \eqref{ec2.l5} is bounded above by
$$
e^{\ell\kappa I_{\rho}(K^+\rho)} + \ell\kappa  \int_{K^+\rho}^{\infty} I'_{\rho}(z) \exp\Big( \ell[\kappa  I_{\rho}( z) -  I_{\rho^+}(z) + \frac{z}{n^{\alpha}}\Big(\frac{1}{\rho^+} - \frac{1}{z} \Big)^2]\Big) \ dz \ . 
$$

From Lemma \ref{Irho}, item ii), we have that the expression above is bounded by
\begin{equation}
\label{inter2b}
e^{\ell \kappa I_{\rho}(K^+\rho)} +  \frac{2\ell\kappa \rho^+}{\rho} \int_{K^+\rho}^{\infty} I'_{\rho^+}(z) \exp\Big(\ell\big [(16\kappa K^+ -1)I_{\rho^+}(z) +        \frac{z}{n^{\alpha}}\Big(\frac{1}{\rho^+} - \frac{1}{z} \Big)^2 \big]\Big) \ dz \ .
\end{equation}

On the other hand, taking $0< \beta< (32C_0K)^{-1}$ and using that 
$$
-\frac{1}{4}I_{\rho^+(z)} + \frac{z}{n^{\alpha}}\Big(\frac{1}{\rho^+} - \frac{1}{z} \Big)^2 <0,
$$
we conclude that $16\kappa K^+ <1/2$  for $n$ larger enough. In consequence, for some $\hat{\kappa}>1/4$ we can bound \eqref{inter2b} by

$$
e^{\ell \kappa I_{\rho}(K^+\rho)} +  \frac{2\ell\kappa \rho^+}{\rho} 
\int_{K^+\rho^{+}}^{\infty} I'_{\rho^+}(z) \exp\Big(-\ell\hat \kappa I_{\rho^+}( z) \Big) \ dz \ , 
$$
which is equal to
\[
e^{\ell\kappa I_{\rho}(K^+\rho)} + \frac{2\kappa\rho^+}{\hat \kappa \rho}  \exp\Big( {-\ell\hat \kappa   I_{\rho^+}(K^+\rho) \Big)} \  . 
\]
Therefore,  \eqref{ec2.l5} restricted to the interval $[K^+\rho, \infty)$ is bounded above by
\begin{equation}
\label{inter3}
e^{\ell\kappa I_{\rho}(K^+\rho)} + \frac{2\kappa\rho^+}{\hat \kappa \rho} . 
\end{equation}
In view of the elementary inequalities
\begin{align*}
\log E[ e^{X+Y+Z}] &\leq \tfrac{1}{3} \big\{\log E[e^{3X}] + \log E[e^{3Y}]+\log E[e^{3Z}]\big\}, \\
\limsup_{n\to\infty} n^{-1} \log\{a_n + b_n\} &\le
\max\{ \limsup_{n\to\infty} n^{-1} \log a_n , \limsup_{n\to\infty}
n^{-1} \log b_n \},
\end{align*}
to show that \eqref{ec4} is of order $o(n)$ we can deal with each one of the terms in the bounds obtained in \eqref{inter2a}, \eqref{inter1} and \eqref{inter3}, separately. Thus, to conclude the proof, it is enough to use the definition of $I_{\rho}(a)$, given in \eqref{ldprateexp}, together with the fact that $\rho^-_{n,\ell,x,s} \approx \rho_{n,\ell,x,s} \approx \rho^+_{n,\ell,x,s}$ as $n$ goes to infinity, as guaranteed by Lemma \ref{l3}.
\end{proof}

\section{The one-block estimate}
\label{s4}

In this section we prove Lemma \ref{l1}, which is the main result used in the proof of  Theorem \ref{t3}. 
Observe that Lemma \ref{l1} is equivalent to :
\begin{lemma}
\label{l1bis}
For  $\alpha, \delta >0$ satisfying $ 2 \alpha+3 \delta < 2$, we have
\[
\lim_{n \to \infty} n^{-1} \bb E^n\Big[ \Big| \int_0^t \sum_{x \in \bb T_n} F_n(x,s) n^\alpha \big\{  g(\eta_s^n(x)) - \phi(n^\alpha \eta_s^{n,\ell}(x))\big\} ds\Big| \Big] =0
\]
uniformly in $t \in [0,T]$, where $\ell = \ell_n = n^\delta$.
\end{lemma}
Recall that $\phi(\rho)=\frac{\rho}{1 + \rho}$ and that by the definition given in \eqref{etanl} the quantity $n^\alpha \eta^{n,\ell}(x)$ corresponds to the average of particles on a box of size $\ell$ at the right of $x \in \bb T_n$.

The proof of Lemma \ref{l1bis} will be divided into three steps. First, we introduce the spatial average
$$
g^{n,\ell}_x(\eta) = \frac{1}{\ell }\sum_{i = 1}^{\ell} g(\eta^n(x+ i)),
$$
in the place of $g^{n}(\eta^n(x))$. Then $g^{n,\ell}_x(\eta)$ is replaced by
$$
\psi^{n,\ell}_x(\eta)=\bb E_{\mu^n_{\rho}}[g(\eta^n(x))|n^{\alpha} \eta^{n, \ell}(x)]
$$
and finally we show that $\psi^{n,\ell}_x(\eta)$ is close to $\phi(n^\alpha \eta_s^{n,\ell}(x))$. 
Observe that $\psi^{n,\ell}_x$ does not depend on $ \rho$. 

Besides the entropy inequality and Feynman-Kac's formula, the main tool used in this section is the so-called {\em spectral gap inequality}. Before stating the spectral gap inequality, we need to introduce some definitions. For $\ell \in \bb N$ and  $k \in \bb N_0$ define $\Lambda_\ell= \{1,\dots, \ell\}$, $\Sigma_\ell = \bb N_0^{\Lambda_\ell}$ and
\[
\Sigma_{k,\ell}= \big\{ \eta \in \Sigma_\ell \ ; \  \sum_{i=1}^\ell \eta(i)=k\big\}.
\]
Let $\mu_{k,\ell}$ be the uniform measure on $\Sigma_{k,\ell}$ and let us denote by $\<\cdot,\cdot\>_{k,\ell}$ the inner product in $L^2(\mu_{k,\ell})$: notice that $\mu_{k,\ell}$ is also the restriction of $\mu^n_\rho$ to $\Sigma_{k,\ell}$. 
 For $f:\Sigma_{k,\ell} \to \bb R$ let $L_\ell f: \Sigma_{k,\ell} \to \bb R$ be given by
\[
L_\ell f(\eta) = \sum_{\substack{x, y \;  \in \; \Lambda_\ell\\ |y-x|=1}} g(\eta(x)) \big(f(\eta^{x,y})-f(\eta)\big).
\] 
\begin{proposition}[Spectral gap inequality]
\label{SG}
There exists a universal constant $\kappa_0$ such that 
\[
\<f,f\>_{k,\ell} \leq \kappa_0(\ell+k)^2 \<f,-L_\ell f\>_{k,\ell}
\]
for any $k, \ell \geq 0$ and any function $f:\Sigma_{k,\ell} \to \bb R$ such that $\int f d\mu_{k,\ell}=0$.
\end{proposition}

This proposition was proved in \cite{Mor} for the zero-range process evolving on the complete graph and extended to finite subsets of $\bb Z^d$ using the so called {\em path lemma}. In our one-dimensional situation, a proof can be obtained by coupling with the exclusion process.

For $x \in \bb T_n$ and $\ell < n \in \bb N$, define $\Lambda_{\ell}(x) = \{x+1,\dots,x+\ell\}$. 
Notice that objects like $\mu_{k,\ell}$ or $L_\ell$ can be defined in $\Lambda_\ell(x)$ in a canonical way.
Finally, let us denote by  $\<\cdot,\cdot\>_{\rho}$ the inner product in $L^2(\mu^n_\rho)$ and
\[
L_n f(\eta) = \sum_{\substack{x, y \in \mathbb T_n\\ |y-x|=1}} g(\eta(x)) \big(f(\eta^{x,y})-f(\eta)\big).
\] 

 From now on we will consider $F_n: \bb T_n \times [0,T] \to \bb R$ uniformly bounded by a constant $K$ (as obtained in Lemma \ref{l2.1} for $F_n$ defined in \eqref{Fn}) and 
suppose that the sequence of measures $\{\mu^n; n \in \bb N\}$ fulfilled the hypothesis of Theorem \ref{t3}, in particular, $\mu_\epsilon^n \preceq \mu^n \preceq \mu_{\epsilon^{-1}}^n$. Moreover, we will write $\ell$ for $\ell_n =n^{\delta}$. Let us define 
$$
V^{n,\ell}_x(\eta) =g^{n,\ell}_x(\eta) - \psi^{n,\ell}_x(\eta).
$$

Now we proceed to verify the steps involved in the proof of Lemma \ref{l1bis}, beginning with step two.
\begin{lemma}
\label{1BLemma}
For $\alpha, \delta >0$ satisfying $2\alpha + 3\delta <2$, we have
\begin{equation}
\label{1B}
\lim_{n \to \infty}  \bb E_{\mu^n}\Big[\Big|\int_{0}^{t} \frac{1}{n}\sum_{x \in \bb T_n} F_n(x,s) \  n^{\alpha} V^{n, \ell}_x(\eta_s) \ ds \Big|\Big]=0,
\end{equation}
uniformly in $t\in [0,T]$, where $\ell = \ell_n = n^\delta$.
\end{lemma}

The following result deals with large densities in Lemma \ref{1BLemma} .
\begin{lemma} 
\label{cutoff}
Let $\alpha, \delta>0$ and $M > \epsilon^{-1}$. Then, 
\begin{equation*}
\lim_{n \to \infty}\bb E_{\mu^n}\Big[\Big|\int_{0}^{t} \frac{1}{n}\sum_{x \in \bb T_n} F_n(x,s) n^{\alpha}  V^{n, \ell}_x(\eta_s) \mb 1_{\eta^{n,\ell}_s(x) \geq M } \ ds \Big|\Big]=0.
\end{equation*}
uniformly in $t\in [0,T]$, where $\ell = \ell_n = n^\delta$.

\end{lemma}
\begin{proof} 
Since $\|V^{n, \ell}_x\|_{\infty} \leq 2$, we can bound the expectation above by
$$
2 K \ n^{\alpha - 1}\ \bb E_{\mu^n}\Big[\int_{0}^{t} \sum_{x \in \bb T_n}  \mb 1_{\eta^{n,\ell}_s(x) \geq M } \ \ ds \Big] .
$$
Attractiveness of the zero range process
(see Proposition \ref{p1.5}) 
permit us to bound the last expression by $2 K T n^{\alpha} \mu^n_{\epsilon^{-1}} \big(\eta^{n,\ell}(0) \geq M \big)$.
The assertion of the lemma follows by applying the concentration inequality \eqref{ec3.1}.
\end{proof}
 
\begin{proof}[Proof of Lemma \emph{\ref{1BLemma}}]
In view of the preceding lemma, it is enough to prove the desired result for $V^{n,\ell, M}_x : =V^{n,\ell}_x \mb 1_{\eta^{n,\ell}_s(x) \leq M } $ instead of  $V^{n,\ell}_x$, whenever $M>\epsilon^{-1}$. In fact, thanks to the uniform boundedness of $F_n$, it suffices to prove, for each $x \in \mathbb T_n$, 
\begin{equation*}
\lim_{n \to \infty}  \bb E_{\mu^n}\Big[\Big|\int_{0}^{t}  \  n^{\alpha} V^{n, \ell, M}_x(\eta_s) \ ds \Big|\Big]=0.
\end{equation*}

As a consequence of the entropy inequality and the fact that $H(\mu^n | \mu^n_{\rho})=O(n)$ for any $\rho>0$,
there exists a positive constant $C$ for which the expectation in (\ref{1B}) is bounded above by
\begin{equation*}
\frac{C}{\lambda}+\frac{1}{\lambda n} \log \bb E_{\mu_{\rho}^n}\Big[\exp\Big\{\lambda n^{\alpha + 1} \Big| \int_{0}^{t}  V^{n,\ell,M}_x(\eta_s) \ ds \Big| \Big\}\Big],
\end{equation*}
for every $\lambda > 0$. Since $e^{|z|}\leq e^z + e^{-z}$, in order to conclude the proof
will be enough to show
\begin{equation}
\label{1Blog}
\limsup_{n \to \infty} \frac{1}{\lambda n } \log \bb E_{\mu_{\rho}^n}\Big[\exp\Big\{\lambda n^{\alpha + 1} \int_{0}^{t}   V_x^{n,\ell,M}(\eta_s)\ ds \Big\}\Big] \leq 0,
\end{equation}
for any $\lambda>0$. By Feynman-Kac's formula, the logarithm in the preceding line is bounded by $t$ times the largest eigenvalue of the operator
\begin{equation*}
n^{2 + 2\alpha} L_n +\lambda n^{\alpha+1}  V_x^{n,\ell,M}.
\end{equation*}
Using the variational formula for the largest eigenvalue of an operator in $L^2(\mu^n_{\rho})$, it can be seen that the left hand side of (\ref{1Blog}) is bounded by
\begin{equation}
\label{1Bvar}
t n^{\alpha} \sup_{\<f,f \>_{\rho} = 1} \Big\{ \Big\<  V_x^{n,\ell,M},f^2\Big\>_{\rho} - \frac{n^{1 + \alpha}}{\lambda}\Big\<f, -L_{n} {f}\Big\>_{\rho}\Big\}.
\end{equation}
It is not difficult to see that
$$
\Big\< V_x^{n,\ell,M},f^2\Big\>_{\rho} = \Big\< W^{n,\ell,M}_0,\overline{f_{\ell}^2} \Big\>_{\rho}  \quad \text{with} \quad  \overline{f_{\ell}^2} = E_{\mu^n_{\rho}}\Big[\frac{1}{\ell}\sum_{i =1}^{\ell} \tau_{-x-i} f^2\Big| \Delta_{n,\ell}\Big],
$$
where 
$$
W^{n,\ell,M}_0(\eta) = \{g(\eta_0)-\psi^{n,\ell}_0(\eta) \}  \mb 1_{\eta^{n,\ell}(0) \leq M} 
$$  
and 
$E_{\mu^n_{\rho}} [ \ \cdot \  | \Delta_{n,\ell} ]$ denotes the conditional expectation in a box of size $\ell$ and $\tau$ correspond to the spatial right shift.
On the other hand
\begin{align*}
\Big\<\sqrt{\overline{f_{\ell}^2}}, (-L_{\ell}) \sqrt{\overline{f_{\ell}^2}}\ \Big\>_{\rho} \leq \frac{\ell}{n} \<{f}, -L_{n}{f}\>_{\rho}.
\end{align*}
Therefore, the supremum in (\ref{1Bvar}) is bounded by 
\begin{equation}
\label{1Blocal}
\sup_{\<f,f \>_{\rho} = 1} \Big\{ \<   W^{n,\ell,M}_0,{\overline{f_{\ell}^2}}\>_{\rho} - \frac{n^{2 + \alpha}}{\ell \lambda}\big\<\sqrt{\overline{f_{\ell}^2}}, (-L_{\ell}) \sqrt{\overline{f_{\ell}^2}} \ \big\>_{\rho}\Big\}.
\end{equation}

For every function $f$ such that $\<f,f \>_{\rho} = 1$, defining 
$c_{k,\ell}(f):= E_{\mu^n_{\rho}} \big[ \overline{f_{\ell}^2} \mb 1_{\Sigma_{k,\ell}} \big]$ 
for $k \in \bb N$,  we have that
  $\overline{f_{k,\ell}^2}(\eta):= c_{k,\ell}(f)^{-1}\mu^n_{\rho}(\Sigma_{k,\ell}) \overline{f_{\ell}^2}(\eta)$ satisfies $\mu_{k,\ell}(\overline{f_{k,\ell}^2})=1$.
The expression inside brackets in \eqref{1Blocal} is bounded above by
\begin{equation}
\label{e1}
\sum_{k=0}^{M n^\alpha \ell} c_{k,\ell}(f) \Big\{ \< W^{n,\ell,M}_0,{\overline{f_{k,\ell}^2}}\>_{k,\ell} - \frac{n^{2+ \alpha}}{\ell \lambda}\<\sqrt{\overline{f_{k,\ell}^2}}, (-L_{\ell}) \sqrt{\overline{f_{k,\ell}^2}}\ \>_{k,\ell} \ \Big\}.
\end{equation}

Using Rayleigh expansion (see Theorem A3.1.1, in \cite{KipLan}), we see that the above expression into brackets is less than or equal to
\begin{equation*}
\<  W^{n,\ell,M}_0\>_{k,\ell} + 
 \frac{ \big( \frac{\ell \lambda}{n^{2+ \alpha}} \big)^2 \<  W^{n,\ell,M}_0,  (-L_{\ell})^{-1} W^{n,\ell,M}_0\>_{k,\ell}}
 {
 1
- 2 \| W^{n,\ell,M}_0 \mb 1_{\Sigma_{k,\ell}} \|_{\infty} \frac{\ell \lambda}{n^{2+ \alpha}} \Gamma_{k,\ell} } , 
\end{equation*}
where $\Gamma_{k,\ell}$ is the magnitude of the spectral gap of $L_{\ell}$ restricted to ${\Sigma_{k,\ell}}$. Since $\sum_{k=0}^{\infty} c_{k,\ell} = 1$ and by definition $\<  W^{n,\ell,M}_0\>_{k,\ell}=0$ for $k \geq 0 $, we have that \eqref{e1}  is bounded by
$$
\sup_{k \leq M n^{\alpha} \ell} \Big\{
 \frac{\<  W^{n,\ell,M}_0,   W^{n,\ell,M}_0\>_{k,\ell}}{\frac{n^{2 + \alpha }}{l\lambda} \Gamma_{k,\ell}-2 ||W^{\ell,M}_0\mb 1_{\Sigma_{k,\ell}} ||_{\infty} } 
  \Big\} \ .
$$
According to the spectral gap inequality stated in Proposition \ref{SG}, and using that $||W^{n,\ell,M}_0||_{\infty} \leq 2$, we can see that the previous expression is bounded by
$$
\frac{8 \lambda \ell^3 (1+Mn^{\alpha})^2 }{n^{2+ \alpha}}.
$$
Therefore,  the expression in the left hand side of \eqref{1Blog} is bounded by a constant times 
$
n^{2\alpha + 3\delta -2},
$
which in view of the conditions imposed on $\alpha$ and $\delta$, concludes the proof.
\end{proof}
The next two lemmas correspond to steps one and three in the proof of Lemma \ref{l1bis}.  
\begin{lemma}
Let $\alpha>0$ and $0<\delta<1$ . Then, 

\begin{equation}
\lim_{n \to \infty}  \bb E_{\mu^n}\Big[\Big|\int_{0}^{t} \frac{1}{n}\sum_{x \in \bb T_n} F_n(x,s) n^{\alpha} \{g(\eta_s^{n}(x)) - g_x^{n, \ell}(\eta_s) \} \ ds \Big|\Big]=0.
\end{equation}
uniformly in $t\in [0,T]$, where $\ell = \ell_n = n^\delta$.
\end{lemma}
\begin{proof}
Relying on the entropy inequality and Feymann-Kac formula, as at the first part of the proof of Lemma \ref{1BLemma}, we can see that it suffices to show 
{\small
\begin{equation}\label{g2}
\limsup_{n \to \infty} \int_0^t n^{\alpha}
 \sup_{\< f , f \>_{\rho} = 1} \Big\{ \Big\< \frac{1}{n}\sum_{x \in \bb T_n} F_n(x,s) 
\{g(\eta_s^{n}(x)) - g_x^{n, \ell}(\eta_s) \} ,
 f^2\Big\>_{\rho} - \frac{n^{1 + \alpha}}{\lambda}\Big\<f, -L_{n}{f}\Big\>_{\rho}\Big\} \ ds \leq 0 \ .
\end{equation}
}
for any $\lambda>0$.
To avoid cumbersome notation, we will write $\eta_x$ instead of $\eta_s^n(x)$, omitting the super-index of $\eta_s^n$ as well as the time parameter.

After a change of variables we note that 
$$
\mathbb E_{\mu^n_{\rho}}
\big[(g(\eta_x) - g(\eta_{x+i}) )f^2 \big] 
=
\mathbb E_{\mu^n_{\rho}}
\big[ g(\eta_x) (f^2(\eta_{x+i}) - f^2(\eta_x)) \big] ,
$$
therefore, using Young's inequality, we can bound from above the expression into braces in \eqref{g2} by
\begin{equation}\label{young} 
\frac{2\beta}{ n } \sum_{x \in \bb T_n} \ F^2_n(x,s) +
 \frac{1}{2 n \beta}\sum_{x \in \bb T_n} \ \frac{1}{\ell}\sum_{i = 1}^{\ell} \mathcal D_{x,x+i}(f)  -
\frac{n^{1+\alpha}}{2\lambda}  \sum_{x \in \bb T_n} \mathcal D_{x,x+1}(f) ,
\end{equation}
for any $\beta>0$, where 
$
\mathcal D_{x,y}(f) = \mathbb E_{\mu^n_{\rho}}
\big[ g(\eta(x)) \{f(\eta^{x,y}) - f(\eta) \}^2 \big] .
$
Since
$$
\mathcal D_{x,x+i}(f)  \leq i \sum_{j = 1}^{i}  \mathcal D_{x+j-1,x+j}(f),  
$$
we can bound expression \eqref{young} by
$$
\frac{2\beta}{ n } \sum_{x \in \bb T_n} \ F^2_n(x,s) +
\Big( \frac{\ell^2}{2 n \beta}  - \frac{n^{1+\alpha}}{2\lambda} \Big)\sum_{x \in \bb T_n} \mathcal D_{x,x+1}(f). 
$$
The proof is concluded by taking $\beta=\frac{\ell^2 \lambda}{n^{2+\alpha}}$ and noting that

$$
 \frac{\lambda \ell^2}{2n^{2}}  \int_0^t  \frac{1}{n} \sum_{x} F^2_n(x,s) \ ds.
$$
goes to zero as $n$ goes to infinity.
\end{proof}

\begin{lemma} 
Let $\alpha, \delta>0$ and $M > \epsilon^{-1}$. Then, 
\begin{equation}
\label{l4.6}
\lim_{n \to \infty}\bb E_{\mu^n}\Big[\Big|\int_{0}^{t} \frac{1}{n}\sum_{x \in \bb T_n} F_n(x,s) n^{\alpha}  
\{ \psi^{n,\ell}_x(\eta) - \phi(n^\alpha \eta_s^{n,\ell}(x)) \}  \ ds \Big|\Big]=0.
\end{equation}
uniformly in $t\in [0,T]$, where $\ell = \ell_n = n^\delta$.
\end{lemma}
\begin{proof} Explicit calculations give
\begin{equation}
\label{equivensem}
\psi^{n,\ell}_x(\eta) = 
\frac{n^{\alpha}\eta^{n,\ell}(x)}{1 - \frac{ 1}{\ell } + n^{\alpha}\eta^{n,\ell}(x) }.
\end{equation}

Thus, the expression into braces in \eqref{l4.6} is bounded from above by $ [\ell (1 + n^{\alpha}\eta^{n,\ell}(x))]^{-1}$. Reasoning in the same way as in the proof of Lemma \ref{cutoff}
we can bound the expectation above by
 $$
 \ell^{-1} K \ t n^{\alpha} \bb E_{\mu^n_{\epsilon^{-1}}}\Big[ (1 + n^{\alpha} \eta^{n, \ell}(0))^{-1}\Big].
 $$
The assertion of the lemma follows by considering separately the expectation on the sets $\{\eta^{n,\ell}_s(0) \geq 1/2\epsilon\}$ and $\{\eta^{n,\ell}_s(0) < 1/2\epsilon \}$. It is directly seen that the first term behaves as $\ell^{-1}$, while concentration inequality \eqref{ec3.1} implies that the second part goes exponentially fast to zero as $n$ goes to infinity.
\end{proof}

\section{Concentration inequalities and large deviations}
\label{concentration}

In this section we derive concentration inequalities that are needed to prove Theorem \ref{t3}. The estimates are not completely standard due to the increasing density of particles. Therefore, we need to derive non-asymptotic large deviations upper bounds for triangular arrays of independent particles. 

Let $X$ be a random variable with distribution $\Geom(\theta)$: geometric distribution of success probability $\theta$. Notice that $\rho := E[X] = \frac{1 - \theta}{\theta}$ and denoting by $\mc M_\rho(\lambda): = E[e^{\lambda X}]$ the moment generating function of $X$, we have that
\[
\mc M_\rho(\lambda) = \frac{1}{1-\rho(e^\lambda-1)}
\]
for $\lambda < \ln(\frac{1+\rho}{\rho})$ and $\mc M_\rho(\lambda) = + \infty$ otherwise. 

Let us define as well  the large deviations rate function associated to geometric distributions of mean $\rho$:
\begin{equation*}
\mc I_\rho(a) := \sup_{\lambda \in \bb R} \big\{ \lambda a - \log \mc M_\rho(\lambda)\big\} = a \log \frac{a(1+\rho)}{\rho(1+a)} - \log \frac{1+a}{1+\rho}.
\end{equation*}

Fix $0<\alpha<1$ and let $\{X_1^n,\dots,X_\ell^n\}$ be a sequence of independent random variables, such that $X_i^n$ has distribution $\Geom(\frac{\rho_i n^\alpha}{1+\rho_i n^\alpha})$. 
Cr\'amer's method allow us to obtain exponential bounds on the tail probabilities of $S_\ell^n =: X_i^n+\dots X_\ell^n$:

Define $\rho^+ = \max\{\rho_1,\dots,\rho_\ell\}$ and $\rho^-=\min\{\rho_1,\dots,\rho_\ell\}$. Notice that the value of $\lambda$ that realises the supremum in the definition of $\mc I_\rho(a)$ is positive if $a > \rho$ and negative if $a < \rho$. Therefore, we get the bounds 
\begin{equation*}
\tfrac{1}{\ell} \log P( S_\ell^n \geq a n^\alpha \ell) \leq -  \mc I_{\rho^{+}n^\alpha}(a n^\alpha) \quad \text{for}\quad a  \geq \rho^+
\end{equation*}
and 
\begin{equation*}
\tfrac{1}{\ell} \log P( S_\ell^n \leq a n^\alpha \ell) \leq - \mc I_{\rho^{-}n^\alpha}(a n^\alpha) \quad \text{for}\quad 0 < a \leq \rho^-
\end{equation*}

A simple application of L'Hospital's rule shows that 
\begin{equation*}
\lim_{n \to \infty} \mc I_{\rho n^{\alpha}}( a n^{\alpha})=   \tfrac{a}{\rho} -\log \tfrac{a}{\rho} -1.
\end{equation*}

Observe that the right hand side of the last line coincides with the large deviations rate function associated to exponential distributions of mean $\rho$: 
$$
I_\rho(a) = : \sup_{\lambda \in \bb R} \big\{ \lambda a - \log M_\rho(\lambda) \big\} 
$$
where 
$
M_\rho(\lambda) =(1-\rho \lambda)^{-1}
$
for $\lambda < {1}/{\rho}$ and $M_\rho(\lambda) =+\infty \ $ otherwise.
Indeed, this is consistent with the well known fact that 
if $X^n$ has distribution $\Geom(\frac{\rho n^\alpha}{1+\rho n^\alpha})$, then $n^{-\alpha} X^n$ converges to an exponential distribution of mean ${\rho}$. 

The following result provides explicit estimates for tail probabilities of $S_\ell^n$, in terms of the large deviations rate function $I_{\rho}(\cdot)$.

\begin{lemma}
\label{l3.1}
Let $\{\rho_1,\dots,\rho_\ell\}$ be a sequence of positive numbers and let $n \in \bb N$ be fixed. Let $\{X_1^n,\dots,X_\ell^n\}$ be a sequence of independent random variables and assume that $X_i^n$ has distribution $\Geom(\frac{\rho_i n^\alpha}{1+\rho_i n^\alpha})$. Define $S_\ell^n =: X_1^n+\dots+X_\ell^n$, $\rho^+ = \max\{\rho_1,\dots,\rho_\ell\}$ and $\rho^- = \min\{\rho_1,\dots,\rho_\ell\}$. Then, 
\begin{equation}
\label{ec3.1}
\tfrac{1}{\ell}\log P(S_\ell^n \geq \ell a n^\alpha) \leq - I_{\rho^+}(a) + \frac{a}{n^\alpha}\Big(\frac{1}{\rho^+} -\frac{1}{a}\Big)^2, \qquad \text{for any} \qquad a \geq \rho^+
\end{equation}
and 
\begin{equation}
\label{ec3.2}
\tfrac{1}{\ell}\log P(S_\ell^n \leq \ell a n^\alpha) \leq - I_{\rho^-}(a) + \frac{a}{n^\alpha}\Big(\frac{1}{\rho^-} -\frac{1}{a}\Big)^2, \qquad \text{for any} \qquad a \leq \rho^-
\end{equation}
where $I_\rho(a) = \frac{a}{\rho} -\log \frac{a}{\rho} -1$ is the large deviations rate function of an exponential distribution of mean $\rho$.
\end{lemma}

\begin{proof}
Most of the proof was done above. The first bound stated on this lemma follows by observing that
$$
\mc I_{\rho n^\alpha}(a n^\alpha) \geq
 I_{\rho}(a)  - \frac{(a - \rho)^2 n^\alpha}{\rho {(1+\rho n^\alpha) (1+an^\alpha)}}   \\
$$
for $a \geq \rho^+$. This estimate is obtained by using  the inequality $-\log(1+x) \geq  -x$ in the first term of the definition of $\mathcal I_{\rho}(\cdot)$ and applying the mean value theorem to the second one. In the very same way we obtain
$$
\mc I_{\rho^{-} n^\alpha}(a n^\alpha)  \geq  I_{\rho^-}(a)  - \frac{  (a - \rho)^2}{a \rho (1+ \rho n^\alpha)}, 
$$
for $a \leq \rho^-$. 

\end{proof}

The following lemmas relate  rates $I_\rho(z)$ for different values of $\rho$ and $z$. 

\begin{lemma}
\label{Irho} 
Let  $0<\rho \leq \rho^+$ and   $K^+= K^+(\frac{\rho^+}{\rho})= (\frac{\rho^+}{\rho})^2$, then:
\begin{itemize}
\item[i)] $ I_\rho(z)\leq \big[\frac{z-\rho}{\rho}\big]^2  \ \ \text{for all} \ \ \frac{\rho}{2} < z. $
\item[ii)]
$ I_\rho(a) \leq 16 K^+ I_{\rho^+} (a)  \ \ \text{for all} \ \ a \geq K^+ \rho . $
\end{itemize}  
\end{lemma}
\begin{proof}The first assertion follows by taking $x =1-z/\rho $ in the inequality $x + x^2 + \ln (1-x) \geq 0$, which is valid for $x< 1/2$. For the second assertion, denote $I_\rho(x) = I(\frac x\rho)$ and define $\psi(x) = \min\{x,x^2\}$. Observe that for any $x \geq 1$,
\[
\tfrac{1}{4} \psi(x-1) \leq I(x) \leq \psi(x-1)
\]
and for any $\lambda \in (0,1]$ and  $x \geq \frac{1}{\lambda^2}$
\[
\psi(x-1) \leq \tfrac{4}{\lambda^2} \psi(\lambda x-1). 
\] 
Then take $\lambda = {\rho}/{\rho^+}$ in the  previous inequalities.
  \end{proof}

\begin{lemma}
\label{Irhopsi2}
Fix  $\rho, \tilde \rho>0$ and $0 <   \kappa \leq \tilde \kappa \leq 1 $ such that $\tilde \kappa \rho > \kappa \tilde \rho $. Consider the function $ f_{\rho,\tilde \rho,\kappa \tilde \kappa}: (0,\infty) \to \bb R$ defined as
$$
 f_{\rho,\tilde \rho,\kappa \tilde \kappa} (z) 
 = \kappa  I_{\rho}( z) -    \tilde \kappa I_{\tilde \rho}(z) 
  = \kappa \Big\{ \Big[\frac{1}{\rho} - \frac{1}{\tilde \rho} \Big] z + \ln\Big[\frac{\rho}{\tilde \rho} \Big] \Big\} \ .  
$$
$ f_{\rho,\tilde \rho,\kappa, \tilde \kappa} $ 
attains its maximum at $z^*= (1- \kappa/ \tilde \kappa)(\rho - \kappa \tilde \rho/  \tilde \kappa)^{-1}   \rho \tilde \rho$. Moreover,
$$
f_{\rho,\tilde \rho, \kappa,\tilde \kappa }(z^*) \leq \kappa \tilde \kappa (\tilde \kappa - {\kappa}/{2})[\tilde \kappa \rho - \kappa \tilde \rho]^{-2}[\rho - \tilde \rho]^2\, .
$$
\end{lemma}

\begin{proof}
To obtain $z^*$ it is enough to note that 
$$
f'_{\rho,\tilde \rho,\kappa,\tilde \kappa}(z)=I'_{\rho}( z)\Big[\kappa -\tilde \kappa \frac{\rho(z-\tilde \rho)}{ \tilde \rho(z-\rho)}\Big]\,, 
$$
a fact that follows from the identity
$ \frac{ I'_{\tilde \rho}(z)}{I'_{\rho}( z)} =    \frac{\rho(z-\tilde \rho)}{\tilde \rho(z -\rho)}$. 
To prove the last assertion of the lemma, use the inequalities $x + x^2 + \ln (1-x) \geq 0$ for $x< 1/2$ and $x + x^2 / 2 + \ln (1-x) \leq 0$ for $x< 1$.
\end{proof}

\section{Discrete approximations of diffusion equations}
\label{s5}
In this section we prove Theorem \ref{t2}. Although our proof is completely analytic, we want to stress that most of the computations were guided by probabilistic arguments, like for example the {\em attractiveness} of the zero-range process, which is a probabilistic counterpart of the strong maximum principle. 

The structure of the proof is also borrowed from the usual way to prove convergence in distribution of stochastic processes: first we prove that the discrete approximations are tight in a convenient functional space, then we show that any limit point is a solution of the hydrodynamic equation \eqref{echid} and finally we obtain a uniqueness criterion for such solutions of \eqref{echid}, which implies uniqueness of the limit point and convergence of the discrete approximations to this unique point.

We start obtaining various properties of the discrete approximations. Then we will use the properties as building blocks for the proof of Theorem \ref{t2}.

Let us recall, for the convenience of the reader, that fixed an initial condition $u_0: \bb T \to [0,+\infty)$, we define $u_0^n: \bb T_n \to [0,\infty)$ by $u_0^n(x) = u_0(\frac{x}{n})$, $x \in \bb T_n$, and  the system of ODE's \eqref{ec1.4.3} is given by
\begin{equation*}
\tfrac{d}{dt} u_t^n(x) = \Delta_n \phi_n(u_t^n(x))
\end{equation*}
with initial data $u_0^n$, where $\phi_n(u) = n^\alpha\phi(n^\alpha u)$ and $\phi(\rho) = \rho/(1+\rho)$.

\subsection*{Uniqueness of solutions}

Notice that $0 \leq \phi_n'(u) \leq n^{2\alpha}$. Therefore, the right-hand side of \eqref{ec1.4.3} is Lipschitz as a function of $u_t^n$ and the conditions for existence and uniqueness of global solutions are fulfilled. For further reference, we state this as a lemma:

\begin{lemma}
\label{l5.1} For any initial condition $u_0^n: \bb T_n \to [0,\infty)$ there exists a unique local solution $\{u_t^n(x); t \in [0,\tau], x \in \bb T_n\}$ of \eqref{ec1.4.3}. This local solution can be extended to a global solution of \eqref{ec1.4.3} in a unique way. The resulting solution is continuous in $t$ and also continuous as a function of the initial condition $u_0^n$.
\end{lemma}

 \subsection*{Strong maximum principle}
 
Let $\{u_t^n(x); t \geq 0, x \in \bb T_n\}$ and $\{v_t^n(x); t \geq 0, x \in \bb T_n\}$  be two solutions of \eqref{ec1.4.3} with initial conditions $u_0^n$ and $v_0^n$, respectively. The following lemma is what is known as the {\em strong maximum principle} for equation \eqref{ec1.4.3}:

\begin{lemma}
\label{l5.2}
Assume that $u_0^n(x) \leq v_0^n(x)$ for any $x \in \bb T_n$. Then  
\[
u_t^n(x) \leq v_t^n(x)
\]
for any $t \geq 0$ and any $x \in \bb T_n$.
\end{lemma}

\begin{proof}
Let us define $\{\delta_t^n(x); t \geq 0, x \in \bb T_n\}$ as
\[
\delta_t^n(x) = v_t^n(x) -u_t^n(x)
\]
for any $t \geq 0$ and any $x \in \bb T_n$.
Notice that for any $u,v \in [0,\infty)$,
\[
\phi_n(v) - \phi_n(u) = \frac{n^{2\alpha}(v - u)}{(1+ n^{\alpha}u)(1+n^{\alpha}v)}.
\]
Therefore,
\begin{equation}
\label{ec5.4}
\tfrac{d}{dt} \delta_t^n(x) = \Delta_n \big\{ H_t^n(x) \delta_t^n(x)\big\},
\: \: \: \: \footnote{
At this point, the lemma has a simple probabilistic proof. Equation \eqref{ec5.4} is the forward Fokker-Planck equation of a inhomogeneous random walk in $\bb T_n$ with the following dynamics: if the walk rests at site $x$ at time $t$, it waits a Poissonian time of instantaneous rate $H_t^n(x)$, at the end of which it jumps to a neighboring site chosen with uniform probability, Since the solution of the Fokker-Planck equation is the space-time distribution of the random walk, $\delta_t^n(x)$ is always non-negative.
}
\end{equation}
where $\{H_t^n(x); t \geq 0, x \in \bb T_n\}$ is given by
\[
H_t^n(x) = \frac{n^{2\alpha}}{(1+ n^{\alpha}u_t^n(x))(1+n^{\alpha}v_t^n(x))}
\]
for any $t \geq 0$ and any $x \in \bb T_n$. Notice that equation \eqref{ec5.4} inherits uniqueness, continuity in time and continuity with respect to the initial condition from the corresponding properties of \eqref{ec1.4.3}, see Lemma \ref{l5.1}. The proof is concluded by showing that 
\begin{equation}
\label{ec5.6}
\delta_0^n(x) \geq 0 \text{ for any } x \in \bb T_n \implies \delta_t^n(x) \geq 0 \text{ for any } t \geq 0 \text{ and any } x \in \bb T_n.
\end{equation}
This is actually a consequence of the fact that $H_t^n(x) > 0$ for any $t \geq 0$ and any $x \in \bb T_n$. We start proving the (seemingly) weaker statement
\begin{equation}
\label{ec5.7}
\delta_0^n(x) \geq \epsilon > 0 \text{ for any } x \in \bb T_n \implies \delta_t^n(x) \geq 0 \text{ for any } t \geq 0 \text{ and any } x \in \bb T_n.
\end{equation}
In this case, define
\[
\tau = \inf\big\{t \geq 0; \min_{x \in \bb T_n} \delta_t^n(x) = 0\big\}.
\]
If $\tau = +\infty$ there is nothing to prove. Assume that $\tau <+\infty$. By continuity $\tau >0$ and there exists $x_0 \in \bb T_n$ such that $\delta_{\tau}^n(x_0) =0$, and $\delta_t^n(x) >0$ for any $0 \leq t< \tau$ and any $x \in \bb T_n$. Therefore $\frac{d}{dt} \delta_{\tau}^n(x_0) \leq 0$ and by \eqref{ec5.4} $\Delta_n\{ H_\tau^n(x_0) \delta_\tau^n(x_0)\} \leq0$. Since $H_\tau^n$ is strictly positive and $\delta_\tau^n(x_0)=0$, we conclude that $\delta_\tau^n(x)=0$ for any $x \sim x_0$. Iterating this argument we conclude that $\delta_\tau^n(x) =0 $ for any $x \in \bb T_n$. By uniqueness of solutions, we conclude that $\delta_t^n \equiv 0$ for any $t \geq \tau$, which proves the weaker property \eqref{ec5.7}. By continuity with respect to the initial condition, taking $\epsilon \to 0$ we see that \eqref{ec5.6} is actually equivalent to \eqref{ec5.7}, which proves the lemma.
\end{proof}

Notice that constant functions in space and time are solutions of \eqref{ec1.4.3}. Therefore, the strong maximum principle has the following corollary:

\begin{corollary}[Weak maximum principle]
\label{c5.1}
For any $t \geq 0$,
\[
\min_{x \in \bb T_n} u_0^n(x) \leq \min_{x \in \bb T_n} u_t^n(x) \leq \max_{x \in \bb T_n} u_t^n(x) \leq \max_{x \in \bb T_n} u_0^n(x)
\]
\end{corollary}

\begin{remark}
Notice that the total mass
\[
\sum_{x \in \bb T_n} u_t^n(x)
\]
is preserved by the evolution. This fact rules out {\em a posteriori} the possibility $\tau <+\infty$, a fact which is not a consequence {\em a priori} of the strong maximum principle. 
\end{remark}

\subsection*{The energy estimate}

Define the {\em energy} of a function $u: \bb T_n \to \bb R$ as 
\[
\mc E_n(u) = \sum_{x \in \bb T_n} n \big\{ \phi_n\big(u(x+1)) - \phi_n\big(u(x)\big)\big\}^2.
\]
Notice that 
\[
\tfrac{d}{dt} \mc E_n(u_t^n) = -\frac{2}{n} \sum_{x \in \bb T_n} \phi_n'(u_t^n(x)) \big( \Delta_n \phi_t^n(x)\big)^2
\]
and the energy of a solution of \eqref{ec1.4.3} is decreasing in time. In particular, at any positive time $t$, the energy of $u_t^n$ is non-negative and bounded above by $\mc E_n(u_0^n)$. For further reference, we state this fact as a lemma:

\begin{lemma}
\label{l5.3}
Let $\{u_t^n(x); t \geq 0, x \in \bb T_n\}$ be a solution  of \eqref{ec1.4.3}. Then
\[
\sum_{x \in \bb T_n} n \big( \phi_t^n(x+1)-\phi_t^n(x)\big)^2 \leq \sum_{x \in \bb T_n} n \big( \phi_0^n(x+1)-\phi_0^n(x)\big)^2
\]
for any $t \geq 0$.
\end{lemma}

The utility of the energy estimate comes from the following estimate, which is just a discrete version of Poincar\'e's inequality:

\begin{lemma}
\label{l5.4}
For any function $u: \bb T_n \to \bb R$ and any $x,y \in \bb T_n$,
\[
\big| \phi_n(u(y)) - \phi_n(u(x))\big| \leq \mc E_n(u)^{1/2} \Big| \frac{y-x}{n}\Big|^{1/2}.
\]
\end{lemma}

\begin{remark}
The proof of this lemma is a direct application of Cauchy-Schwarz inequality, and we omit it.
\end{remark}
\begin{remark}
\label{r5.7}
If we take $u_n(x) = u(\frac{x}{n})$, where $u: \bb T \to \bb  R$ is of class $\mc C^1(\bb T)$, then $\mc E_n(u_n) \to \int \phi'(u(x))^2 u'(x)^2 dx$ and the functions $u_n$ are H\"older continuous of index $\frac{1}{2}$, uniformly in $n$ as soon as this integral is finite.
\end{remark}

\subsection*{Time regularity}
Notice that Remark \ref{r5.7} and Lemma \ref{l5.4} can be combined to show that the discrete approximations $\{u^n_t(x); x \in \bb T_n, t \geq 0\}$ are equicontinuous in space, given that the initial data $u_0$ is of class $\mc C^1(\bb T)$ and $\int \phi'(u(x))^2 u(x)^2 dx <+\infty$. Therefore, it would be nice to show that the discrete approximations are also regular in time. Notice that the time derivative of $u_t^n$ is equal to $\Delta_n \phi_t^n(x)$. Therefore, if we are able to prove that $\Delta_n \phi_t^n(x)$ is uniformly bounded, that would imply that $u_t^n(x)$ is uniformly Lipschitz in time, giving the desired regularity.

Define $\{\psi_t^n(x); x \in \bb T_n; t \geq 0\}$ as
\[
\psi_t^n(x) = \phi_n'(u_t^n(x)) \Delta_n \phi_t^n(x) 
\]
for any $x \in \bb T_n$ and any $t \geq 0$. The function $\psi_t^n$ satisfies the equation
\begin{equation}\label{EL}
\tfrac{d}{dt} \psi_t^n(x) = \phi_n'(u_t^n(x)) \Delta_n \psi_t^n(x) + \frac{\phi_n''(u_t^n(x))}{\phi_n'(u_t^n(x))^2} \psi_t^n(x)^2.
\end{equation}
Notice that $\frac{\phi_n''(u)}{\phi_n'(u)^2} = -2(n^{-\alpha}+u)$. Define 
\[
\Psi_n(t) = \max_{x \in \bb T_n} \big|\psi_t^n(x)\big|, \quad \bar{u}_n= \max_{x \in \bb T_n} u_0^n(x).
\]
Recall that by Corollary \ref{c5.1}, $u_t^n(x) \leq \bar{u}_n$ for any $x \in \bb T_n$ and any $t \geq 0$.
We have that
\[
\tfrac{d}{dt} \Psi_n(t)  \leq 2(n^{-\alpha}+\bar{u}_n) \Psi_n(t)^2.
\]
Integrating the corresponding ODE we conclude that
\[
\Psi_n(t) \leq 2 \Psi_n(0)
\]
whenever
\[
t \leq \frac{1}{ 4 \Psi_n(0)(n^{-\alpha}+\bar{u}_n)}.
\]
Recall that $\phi_n'(u) = (n^{-\alpha}+u)^{-2}$ and recall that $\frac{d}{dt} u_t^n(x) = \phi_n'(u_t^n(x))^{-1} \psi_t^n(x)$. We have proved the following estimate for the time derivative of $u_t^n$:

\begin{lemma}
\label{l5.5}
Let $\{u_t^n(x); t \geq 0, x \in \bb T_n\}$ be a solution of \eqref{ec1.4.3}. Then
\[
\sup_{\substack{x \in \bb T_n \\ 0 \leq s \leq T}} \Big|\tfrac{d}{dt} u_s^n(x) \Big|
\leq 2 (n^{-\alpha}+\bar{u}_n)^2 \bar{\Psi}_n,
\]
where 
\[
\bar{\Psi}_n = \max_{x \in \bb T_n} \frac{\big|\Delta \phi_0^n(x)\big|}{\phi_n'(u_0^n(x))}, \quad \bar{u}_n = \max_{x \in \bb T_n} u_0^n(x)
\]
and
\[
T = \frac{1}{4 \bar{\Psi}_n(n^{-\alpha}+\bar{u}_n)}.
\]
\end{lemma}

\begin{remark}
The time window can be improved making the constant $2$ bigger. However, this procedure will only make approach the constant $4$ in the time window to $1$. Therefore, there is no much gain on it, and we have chosen the form of the lemma stated here for simplicity.
\end{remark}

\begin{remark}
\label{energy2}($\psi$ is $\frac{1}{2}$-H\"older Continuous)
In the same spirit as before, considering the new functional of energy defined by
\[ 
\tilde{\mc E_n}(u) =    \sum_{x \in \bb T_n} n \big( \psi^n_t(x+1) - \psi^n_t(x)\big)^2,
\]
and using \eqref{EL}, we obtain
\begin{eqnarray*}
\tfrac{d}{dt}\tilde{\mc E_n}(u_t^n)
\leq  \frac 1n\sum_{x\in\bb T_n} \Big( \frac{\Delta_n \phi_t^n(x) }{n^{-\alpha} + u_t^n(x)}\Big)^4. 
\end{eqnarray*}

In particular, in view of Lemma \ref{l5.5}, the Laplacian above is bounded for any positive time $0<t< T$, which implies the existence of a constant $K$ such that
\[
\big| \psi^n_t(y) - \psi^n_t(x)\big| \leq K \Big| \frac{y-x}{n}\Big|^{1/2},
\]
for all $0<t<T$ and  $x,y \in \bb T_n$.
\end{remark}

\begin{remark}
If $u_0^n$ is the discretization of a smooth function $u$, then $\bar{\Psi}_n$ is of the order of
\[
\sup_{x \in \bb T} \big| \Delta u(x)\big|.
\]
Therefore, what this lemma is telling us, is that the discrete approximations of the hydrodynamic equation \eqref{echid} remain regular for a small time interval, whenever the initial condition $u_0$ is regular. The size of this time interval is inversely proportional to the supremum of both $u_0$ and $\Delta u_0$.
\end{remark}

\subsection*{The regular case}
Now we have all the ingredients needed in order to prove Theorem \ref{t2} in the case of a regular initial condition $u_0$ bounded away from zero, up to a positive time $T$. To be more precise, we state this as a lemma:

\begin{lemma}
\label{l5.6}
Let $u_0:\bb T \to [0,\infty)$ be such that $\sup_x |\Delta u_0| <+\infty$. Assume as well that $\inf_x u_0(x) =\epsilon >0$. Then there exists a strictly positive time $T$ such that
\[
\lim_{n \to \infty} \sup_{x \in \bb T_n} \sup_{0 \leq t \leq T} \Big| u_t^n(x) - u(t,\tfrac{x}{n})\Big| =0,
\]
where $\{u(t,x); t \geq 0, x \in \bb T\}$ is the solution of the hydrodynamic equation \eqref{echid} with initial condition  $u_0$.
\end{lemma}

\begin{proof}
Since $\Delta u_0$ is bounded, $u_0'$ is Lipschitz and therefore continuous. Since $u_0$ is also bounded away from $0$,  the energies $\mc E_n(u_0^n)$ are bounded in $n$, which implies that the discrete functions $\{\phi_t^n; t \geq 0, n \in \bb N\}$ are uniformly $\frac{1}{2}$-H\"older in space. Since $u_0$ is bounded above, by the weak maximum principle $u_t$ is bounded above by the same constant. Since $\phi_n'$ is bounded below on finite intervals, we conclude that the functions $\{u_t^n; t \geq 0, n \in \bb N\}$ are also uniformly $\frac{1}{2}$-H\"older in space. In time these functions are uniformly Lipschitz over a non-degenerate time interval $[0,T]$, which depends on the upper bound for $\Delta u_0$ and $u_0$. In particular, the family of space-time functions $\{u_t^n; n \in \bb N\}$\footnote{
As usual in numerical analysis, we will identify $u_t^n$ with the corresponding linear interpolation defined in $\bb T$ and we will not make any distinction between them.
}
is equicontinuous in $n$. Therefore, we have just proved the relative compactness of $\{u_t^n; n \in \bb N\}$ with respect to the uniform topology on $[0,T] \times \bb T$.

Let $n'$ be a converging subsequence and let $u_t=\{u(t,x); t \in [0,T], x \in \bb T\}$ be the corresponding limit point. In $\mc C^1([0,T] \times \bb T )$ consider the Sobolev norm
\[
\|f\|_{1,T} = \Big( \int_0^{T} \int_{\bb T} \big(f(t,x)^2+\partial_x f(t,x)^2\big) dx dt \Big)^{1/2}
\]
and let $\mc H_{1}$ be the completion of $\mc C^1(\bb T)$ under this norm. The Sobolev space $\mc H_{1}$ is a Hilbert space and in particular the closed balls in $\mc H_{1}$ are compact with respect to the weak topology on $\mc H_{1}$. Notice that the energy estimate implies the relative compactness of $\{u_t^n; n \in \bb N\}$ in $\mc H_{1}$. Taking a further subsequence if needed, we can assume that $u_t^n$ and $\phi_t^n$ converge in $\mc H_1$ to $u_t$ and $\phi(u_t)$ respectively. Multiplying equation \eqref{ec1.4.3} by a test function $g$, performing the usual integration-by-parts trick and taking limits along the subsequence $n'$, we conclude that
\[
\int u(T,x) g(T,x) dx = \int u_0(x) g(0,x) dx + \int_0^{T} \int \big(u(t,x)\partial_t g(t,x) + \phi(u(t,x)) \Delta g(t,x) \big) dx dt
\]
for any test function $g$. Since $u_t$ belongs to $\mc H_1$, we can undo one of the spatial integration by parts in this expression to conclude that
\begin{multline*}
\int u(T,x) g(T,x) dx = \int u_0(x) g(0,x) dx + \int_0^{T} \int \big(u(t,x)\partial_t g(t,x) -  \\
		-\partial_x \phi(u(t,x)) \partial_x g(t,x) \big) dx dt
\end{multline*}
for any $g \in \mc H_1$. Up to this point what we have accomplished is to show that the limit point $u_t$ is a weak solution with {\em finite energy} of \eqref{echid}, that is, a weak solution that belongs to $\mc H_1$. Uniqueness of energy solutions follows from Oleinik's method. Let us briefly explain this method. Let $u^i=\{u^i(t,x); t \in [0,T], x \in \bb T\}$, $i=1,2$ be two weak solutions with finite energy of the hydrodynamic equation \eqref{echid} with the same initial condition $u_0$. Taking as a test function
\[
g(t,x) = \int_t^{T} \big( \phi(u^2(s,x)) - \phi(u^1(s,x)\big)ds,
\]
we see that 
\[
\int_0^{T} \int \big(u^2(t,x) -u^1(t,x) \big)^2 dx dt + \tfrac{1}{2} \mc E\Big( \int_0^{T} \big(\phi(u^2_t)-\phi(u^1_t)\big)dt \Big)=0,
\]
where we have used the notation $\mc E(f) = \int f'(x)^2 dx$. Uniqueness follows at once. 
\end{proof}
 
\section{Comments and generalizations}

\label{comments}

Our paper can be considered as an initial effort to solve the challenging problem of obtain the fast diffusion equations:
\begin{equation}
\label{A} 
\left\{
\begin{array}{l}
{\displaystyle \partial_t \rho \; =\; \Delta \rho^\gamma  \;\;\;in\;\;\; \mathbb R^n\times [0,T)} \\
{\displaystyle \rho(0,\cdot) \;=\; \rho_0 (\cdot) \;\;\;on\;\;\; \mathbb R^n},
\end{array}
\right.
\end{equation}
in the range of exponents $\gamma < 1$, as scaling limit of interacting particle systems. In particular, we are interested in understand the phenomena of instantaneous and finite time extinction from a microscopical point of view. Characterizations of the phenomena mentioned above for solutions of \eqref{A} in terms of growth conditions on $u_0$ can be found in \cite[and references therein]{dk, Pel,Vaz}. 

One of the main properties of the zero-range process used in the proof of Theorem \ref{t3} is a sharp bound for the largest eigenvalue of the dynamics restricted to a finite box. In \cite{Nag}, the spectral gap inequality for the zero-range process with interaction rate $g(k) = k^\gamma$ for $\gamma \in (0, 1)$ was derived. Therefore, it is reasonable to consider using techniques similar to those in the previous sections to obtain \eqref{A}, for $\gamma \in (0,1)$, as a scaling limits of zero-range processes.  
  Nevertheless, to obtain the equivalence of ensembles (which in the present setting reduce to the explicit formula \eqref{equivensem}) and concentration inequalities turns to be more demanding in this latter case.
  The new difficulties are due to the fact that, in contrast to the case $g(k) = \mathbf 1_{\{k>0\}}$, there exists no formula in terms of elementary functions to express  the partition function associated to the zero-range process with rate function $g(k) = k^\gamma$. The proof of this scaling limits, under periodic conditions, will be the subject of a forthcoming paper.

\section*{Acklowledgements}
F.H.~would like to thank the warm hospitality of the "Universidade Federal do Espirito Santo" where part of this work was done.
F.H. ~ and  F.V.~  thanks CNPq for its support through the grant 473742/2013-6.


\begin{thebibliography}{10}

\bibitem{And}
E. ~Andjel.
\newblock Invariant measures fot the zero-range process.
\newblock{Ann. Probab.}, 10, 525–547, 1982.

\bibitem{bh}
J. G. ~Berryman, and C. J. ~Holland.
\newblock Asymptotic behavior of the nonlinear differential equation 
$n_t = (n^{-1}n_x)x$.
\newblock{ J. Math Phys}, 23, 983–987, 1982.

\bibitem{cv}
E. ~Chasseigne, and J. L. ~Vazquez.
\newblock Theory of extended solutions for fast-diffusion equations
in optimal classes of data. Radiation from singularities.
\newblock{\em Arch. Ration. Mech. Anal.} 164, 2, 133–187, 2002.

\bibitem{dk}
P. ~Daskalopoulos and  C. E. ~Kenig.
\newblock Degenerate Diffusions - Initial Value Problems and
Local Regularity Theory.
\newblock {\em European Mathematical Society}, 2007.

\bibitem{EkhSep}
M. ~Ekhaus and T. ~Sepp\"al\"ainen.
\newblock  Stochastic dynamics macroscopically governed by the porous
medium equation for isothermal flow.
\newblock{\em Ann. Acad. Sci. Fenn., Math.} 21, 309–352, 1996.

\bibitem{ew}
C. L. ~Epstein, and M. I. ~Weinstein.
\newblock A stable manifold theorem for the curve shortening equation.
\newblock{\em Comm. Pure Appl. Math.} 40, 1, 119–139, 1987.

\bibitem{FenIscSep}
S. ~Feng, I. ~Iscoe and T. ~Sepp\"al\"ainen.
\newblock A microscopic mechanism for the porous medium equation.
\newblock{\em Stochastic Processes and their Applications.} 66, 147–182, 1997.

\bibitem{gh}
M. ~Gage, and R. S. ~Hamilton
\newblock The heat equation shrinking convex plane curves.
\newblock{\em J. Differential Geometry,} 23,1, 69–96, 1986.

\bibitem{GonLanTon}
P. ~ Gon{\c{c}}alves, C. Landim and C. ~Toninelli.
\newblock Hydrodynamic limit for a particle system with degenerate rates.
\newblock {\em Ann. Inst. H. Poincar\'e. Probab. Statist.}, 45(4): 887--909, 2009.

\bibitem{HJV}
F. ~ Hern\'andez, M. ~ Jara and F. ~ Valentim.
\newblock Equilibrium fluctuations for a discrete Atlas model.
\newblock {\em Stoch. Process. App.}, 127(3): 783--802, 2017.

\bibitem{JarLanSet}
M. ~ Jara, C. Landim and S. ~Sethuraman.
\newblock Nonequilibrium fluctuations for a tagged particle in one-dimensional sublinear zero-range processes.
\newblock {\em Ann. Inst. H. Poincar\'e. Probab. Statist.}, 49(3): 611--637, 2013.

\bibitem{KipLan}
C. ~Kipnis and C. ~Landim.
\newblock {\em Scaling limits of interacting particle systems}, volume 320 of
  {\em Grundlehren der Mathematischen Wissenschaften [Fundamental Principles of
  Mathematical Sciences]}.
\newblock Springer-Verlag, Berlin, 1999.

\bibitem{lh}
K. E. ~Lonngren, and A. ~Hirose.
\newblock Expansion of an electron cloud.
\newblock{Phys. Lett. A }.59, 285–286, 1976.

\bibitem{OllVarYau}
S. ~Olla, S. R. S. ~Varadhan, and H.T.~Yau.
\newblock Hydrodynamical limit for a Hamiltonian system with weak noise.
\newblock {\em Comm. Math. Phys.}, 155(3):523--560, 1993.

\bibitem{Mor}
B. ~ Morris.
\newblock Spectral gap for the zero range process with constant rate.
\newblock {\em Ann. Probab.}, 34(5):1645--1664, 2006.

\bibitem{Nag}
Y. ~Nagahata.
\newblock Spectral gap for zero-range processes with jump rate
  {$g(x)=x^\gamma$}.
\newblock {\em Stochastic Process. Appl.}, 120(6):949--958, 2010.

\bibitem{Pel}
Y. ~Nagahata.
\newblock The porous media equation, in Applications of nonlinear analysis in the
physical sciences.
\newblock {\em Surveys Reference Works Math. 6 , Pitman, Boston}, 229--241,1981.

\bibitem{Spi}
F. ~Spitzer.
\newblock Interaction of Markov processes.
\newblock {\em adv. in Math.}, 5: 246--290, 1970.

\bibitem{Vaz}
 H.T. ~ V\'azquez.
\newblock Nonexistence of solutions for nonlinear heat equations of fast-diffusion
type.
\newblock {\em J. Math. Pures Appl.}, 71:503--526, 1992.

\bibitem{Yau}
 J, L H.T. ~Yau.
\newblock  Relative entropy and hydrodynamics of Ginzburg-Landau models. 
\newblock {\em Lett. Math.Phys.}, 22(1):63--80, 1991.

\end{thebibliography}
\end{document}